\documentclass[smallextended,nospthms,envcountsect]{svjour3}

\smartqed 

\usepackage{graphicx}
\usepackage{mathptmx}
\usepackage{amssymb}
\usepackage{amsmath}
\usepackage[matrix,arrow,curve,cmtip]{xy}
\usepackage{mathrsfs}
\usepackage{amsthm}

\numberwithin{equation}{section}

\theoremstyle{plain}   
\newtheorem{bigthm}{Theorem}   

\newtheorem{theorem}[equation]{Theorem}  
\newtheorem{cor}[equation]{Corollary}     
\newtheorem{lemma}[equation]{Lemma}         
\newtheorem{prop}[equation]{Proposition}

\theoremstyle{definition}

\theoremstyle{remark}
\newtheorem{remark}[equation]{Remark}

\newtheorem{question}[equation]{Question}

\newcommand{\Spec}{\operatorname{Spec}}
\newcommand{\gr}{\operatorname{gr}}
\newcommand{\TC}{\operatorname{TC}}
\newcommand{\TR}{\operatorname{TR}}
\newcommand{\THH}{\operatorname{THH}}
\newcommand{\Z}{\mathbb{Z}}
\newcommand{\Fp}{\mathbb{F}_p}
\newcommand{\Sm}{\operatorname{Sm}}
\newcommand{\Sch}{\operatorname{Sch}}
\newcommand{\et}{\operatorname{et}}
\newcommand{\Zar}{\operatorname{Zar}}
\newcommand{\eh}{\operatorname{eh}}
\newcommand{\cdh}{\operatorname{cdh}}
\newcommand{\Nis}{\operatorname{Nis}}
\newcommand{\sh}{{\operatorname{sh}}}
\newcommand{\colim}{\operatornamewithlimits{colim}}
\newcommand{\holim}{\operatornamewithlimits{holim}}
\newcommand{\id}{\operatorname{id}}
\newcommand{\pr}{\operatorname{pr}}
\newcommand{\tr}{\operatorname{tr}}
\newcommand{\red}{\operatorname{red}}

\journalname{Mathematische Annalen}

\begin{document}

\title{On the vanishing of negative $K$-groups\thanks{The first author
was supported in part by NSF Grant No.~0901021 and by the JSPS. The
second author received partial support from NSF Grant No.~0306519.}}


\author{Thomas Geisser \and Lars Hesselholt}


\institute{Thomas Geisser \at
University of Southern California, Los Angeles, California\\
\email{geisser@usc.edu}
\and
Lars Hesselholt \at
Nagoya University, Nagoya, Japan \\
\email{larsh@math.nagoya-u.ac.jp}
}

\date{}

\maketitle

\begin{abstract}
We show that for a $d$-dimensional scheme $X$ essentially of finite
type over an infinite perfect field $k$ of characteristic $p > 0$, 
the negative $K$-groups $K_q(X)$ vanish for $q < - d$ provided that
strong resolution of singularities holds over the field $k$.
\end{abstract}

\section*{Introduction}

A conjecture of Weibel~\cite[Question~2.9]{weibel3} predicts that for 
every noetherian scheme $X$, the negative $K$-groups
$K_q(X)$ vanish for $q < - \dim(X)$. It was proved recently by
Corti\~{n}as, Haesemeyer, Schlichting, and
Weibel~\cite[Theorem~6.2]{cortinashaesemeyerschlichtingweibel} that the
conjecture holds if $X$ is essentially of finite type over a field of
characteristic $0$. In this paper, we prove similarly that the
conjecture holds if $X$ is essentially of finite type over an infinite
perfect field $k$ of characteristic $p > 0$ provided that strong
resolution of singularities holds over $k$. The proofs are by
comparison with Connes' cyclic homology~\cite{loday} and the
topological cyclic homology of
B\"{o}kstedt, Hsiang, and Madsen~\cite{bokstedthsiangmadsen},
respectively.

We say that strong resolution of singularities holds over $k$ if for
every integral scheme $X$ separated and of finite type over $k$, there
exists a sequence of blow-ups
$$X_r \to X_{r-1} \to \dots \to X_1 \to X_0 = X$$
such that the reduced scheme $X_r^{\red}$ is smooth over
$k$; the center $Y_i$ of the blow-up $X_{i+1} \to X_i$ is connected
and smooth over $k$; the closed embedding of $Y_i$ in $X_i$ is
normally flat; and $Y_i$ is nowhere dense in $X_i$. Strong resolution
of singularities holds over fields of characteristic zero by
Hironaka~\cite[Theorem~1*]{hironaka}. In general, the field $k$ must
necessarily be perfect. We say that a scheme is essentially of finite
type over $k$ if it can be covered by finitely many affine open
subsets of the form $\Spec S^{-1}A$ with $A$ a finitely generated 
$k$-algebra and $S \subset A$ a multiplicative subset. The following
result was conjectured by Weibel~\cite[Question~2.9]{weibel3}: 

\begin{bigthm}\label{vanishing}Let $k$ be an infinite perfect field of
characteristic $p > 0$ such that strong resolution of singularities
holds over $k$, and let $X$ be a $d$-dimensional scheme essentially of
finite type over $k$. Then $K_q(X)$ vanishes for $q < -d$.
\end{bigthm}

In general, the group $K_{-d}(X)$ is non-zero. For instance, by closed
Mayer-Vietoris, the group $K_{-d}(\partial \Delta_k^{d+1})$ is readily
seen to be an infinite cyclic group. Therefore, the vanishing
result above is optimal. We further show in Theorem~\ref{finitetype}
below that, under the assumption that strong resolution of
singularities holds over all infinite perfect fields of characteristic
$p > 0$, the conclusion of Theorem~\ref{vanishing} is true for any
scheme $X$ of finite type over any field of characteristic $p$.

To prove Theorem~\ref{vanishing}, we consider the cyclotomic trace map
$$\tr \colon K(X) \to \{ \TC^n(X;p) \}$$
from the non-connective Bass complete $K$-theory spectrum of $X$ to
the topological cyclic homology pro-spectrum of $X$ and define the
pro-spectrum $\{ F^n(X) \}$ to be the level-wise mapping fiber;
compare~\cite[Section~1]{gh5}. Then $\{ F^n(-) \}$ defines a presheaf of
pro-spectra on the category $\Sch/X$ of schemes separated and of
finite type over $X$. We first show that, in the situation of
Theorem~\ref{vanishing}, this presheaf satisfies descent with respect
to the $\cdh$-topology of Voevodsky~\cite{suslinvoevodsky}.

\begin{bigthm}\label{haesemeyer}Let $k$ be an infinite perfect field
of positive characteristic $p$ such that strong resolution of
singularities holds over $k$, and let $X$ be a scheme essentially of
finite type over $k$. Then for all integers $q$, the canonical map
$$\{ F_q^n(X) \} \to
\{ \mathbb{H}_{\cdh}^{-q}(X, F^n(-)) \}$$
is an isomorphism of pro-abelian groups.
\end{bigthm}

We remark that if the dimension of $X$ is zero then
Theorem~\ref{haesemeyer} reduces to the statement that the map
$\{ F_q^n(X) \} \to \{ F_q^n(X^{\red}) \}$ induced by the canonical
inclusion is an isomorphism of pro-abelian groups. This statement, in
turn, is a special case of the general fact that $\{ F_q^n(-) \}$ is
invariant with respect to nilpotent extensions of unital associative
$\Fp$-algebras~\cite[Theorem~B]{gh5}. It would be very interesting to
similarly extend Theorem~\ref{haesemeyer} to a statement valid for all
unital associative $\Fp$-algebras.

To prove Theorem~\ref{haesemeyer} we generalize of a theorem
of Corti\~{n}as, Haesemeyer, Schlichting, and
Weibel~\cite[Theorem~3.12]{cortinashaesemeyerschlichtingweibel} to
show that if strong resolution of singularities holds over the
infinite perfect field $k$, then a presheaf of pro-spectra satisfies
$\cdh$-descent for schemes essentially of finite type over $k$,
provided that it takes infinitesimal thinkenings to weak equivalences
and finite abstract blow-up squares to homotopy cartesian squares, and
provided further that the individual presheaves of spectra satisfy
Nisnevich descent and take squares defined by blow-ups along regular
embeddings to homotopy cartesian squares. The presheaf $\{ F^n(-) \}$
satisfies all four properties according to theorems of
Thomason~\cite[Theorem~2.1]{thomason} 
and~\cite[Theorem~10.8]{thomasontrobaugh}, Blumberg and
Mandell~\cite[Theorem~1.4]{blumbergmandell}, and the
authors~\cite[Theorem~B and~D]{gh5}. Hence, Theorem~\ref{haesemeyer}
follows.

In the situation of Theorem~\ref{vanishing}, every $\cdh$-covering of $X$
admits a refinement to a $\cdh$-covering by schemes essentially smooth
over $k$. Together with a cohomological dimension result of Suslin and
Voevodsky~\cite[Theorem~12.5]{suslinvoevodsky} this shows that the
group $\mathbb{H}_{\cdh}^{-q}(X,K(-))$ vanishes for $q <
-d$. Therefore, in view of Theorem~\ref{haesemeyer}, to prove
Theorem~\ref{vanishing}, it suffices to prove the following result.

\begin{bigthm}\label{tctheorem}Let $k$ be a perfect field of positive
characteristic $p$ such that resolution of singularities holds over
$k$, and let $X$ be a $d$-dimensional scheme essentially of finite
type over $k$. Then the canonical map
$$\{ \TC_q^n(X;p) \} \to 
\{ \mathbb{H}_{\cdh}^{-q}(X, \TC^n(-;p)) \}$$
is an isomorphism of pro-abelian groups for $q < -d$, and an
epimorphism of pro-abelian groups for $q = -d$.
\end{bigthm}

We note that Theorem~\ref{tctheorem} uses the weaker assumption
that resolution of singularities holds over $k$: Every integral
$k$-scheme separated and of finite type admits a proper bi-rational
morphism $p \colon X' \to X$ from a smooth $k$-scheme. We expect the
map in the statement of Theorem~\ref{tctheorem} to be an isomorphism
of pro-abelian groups for  $q \leqslant -d$, and an epimorphism for $q
= -d+1$.

To prove Theorem~\ref{tctheorem}, we take advantage of the fact that
topological cyclic homology, as opposed to $K$-theory, satisfies
\'{e}tale descent. This implies that, to prove Theorem~\ref{tctheorem},
we may replace the $\cdh$-topology by the finer $\eh$-topology
defined in~\cite[Definition~2.1]{geisser2} to be the smallest
Grothendieck topology on $\Sch/X$ for which both \'{e}tale and
$\cdh$-coverings are coverings. The proof of Theorem~\ref{tctheorem} is
then completed by a careful cohomological analysis of the
$\eh$-sheaves $a_{\eh}\TC_q^n(-;p)$ associated with the presheaves of
homotopy groups $\TC_q^n(-;p) = \pi_q\TC^n(-;p)$ in combination with
the following cohomological dimension result.

\begin{bigthm}\label{ehpcdtheorem}Let $X$ be a scheme essentially of
finite type over a field $k$ of characteristic $p > 0$. Then
the $p$-cohomological dimension of $X$ with respect to the 
$\eh$-topology is less than or equal to $\dim(X) + 1$.
\end{bigthm}

We remark that Theorems~\ref{haesemeyer} and~\ref{tctheorem}
both concern pre-sheaves of pro-spectra. Indeed, we do not know whether
the analog of Theorem~\ref{tctheorem} holds if the pre-sheaf of
pro-spectra $\{ \TC^n(-;p) \}$ is replaced by the pre-sheaf of spectra
$\TC(-;p)$ given by the homotopy limit. Our recent paper~\cite{gh5}
was written primarily with the purpose of proving
Theorem~\ref{haesemeyer} above.

Let $X$ be a noetherian scheme. We define $\Sch/X$ to be the category
of schemes separated and of finite type over $X$ and denote by
$a_{\tau} \colon (\Sch/X)^{\wedge} \to (\Sch/X)_{\tau}^{\sim}$ and
$i_{\tau} \colon (\Sch/X)_{\tau}^{\sim} \to (\Sch/X)^{\wedge}$ the
sheafification functor and the inclusion functor, respectively,
between the categories of presheaves and $\tau$-sheaves of sets. We
further denote by $\alpha \colon (\Sch/X)_{\eh} \to (\Sch/X)_{\et}$ the
canonical morphism of sites.

\section{$\cdh$-descent}\label{haesemeyersection}

In this section, we formulate and prove a generalization
of~\cite[Theorem~3.12]{cortinashaesemeyerschlichtingweibel} to
pre-sheaves of pro-spectra. We apply this theorem to prove
Theorem~\ref{haesemeyer} of the introduction. We first recall some
definitions.

Let $X$ be a noetherian scheme, and let $F(-)$ be a presheaf of
fibrant symmetric spectra on $\Sch/X$. If $\tau$ is a Grothendieck
topology on $\Sch/X$ that has enough points, we define the
hypercohomology spectrum 
$\mathbb{H}_{\tau}^{\boldsymbol{\cdot}}(X,F(-))$ to be the
Godement-Thomason construction~\cite[Definition~1.33]{thomason} of the site
$(\Sch/X)_{\tau}$ with coefficients in the presheaf $F(-)$; see
also~\cite[Section~3.1]{gh}. We recall from~\cite[Proposition~3.1.2]{gh} that,
in this situation, there is a conditionally convergent spectral
sequence
$$E_{s,t}^2 = H_{\tau}^{-s}(X,a_{\tau}F_t(-)) \Rightarrow
\mathbb{H}_{\tau}^{-s-t}(X,F(-))$$
from the sheaf cohomology groups of $X$ with coefficients in the
$\tau$-sheaf on $\Sch/X$ associated with the presheaf $F_t(-) =
\pi_t(F(-))$ and with abutment the homotopy groups
$\mathbb{H}_{\tau}^{-q}(X,F(-)) = \pi_q
\mathbb{H}_{\tau}^{\boldsymbol{\cdot}}(X,F(-))$.

We next recall the $\cdh$-topology on $\Sch/X$
from~\cite{suslinvoevodsky}. We say that the cartesian square of
$X$-schemes
$$\xymatrix{
{ Z' } \ar[r]^{i'} \ar[d]^{p'} &
{ Y' } \ar[d]^{p} \cr
{ Z } \ar[r]^{i} &
{ Y } \cr
}$$
is an~\emph{abstract blow-up square} if $i$ is a closed immersion and
$p$ is a proper map that induces an isomorphism of
$Y' \smallsetminus Z'$ onto $Y \smallsetminus Z$. We say that the
square is a~\emph{finite abstract blow-up square} if it is an abstract
blow-up square and the proper map $p$ is finite. We say that the
square is an elementary Nisnevich square if $i$ is an open immersion
and $p$ is an \'{e}tale map that induces an isomorphism of 
$Y' \smallsetminus Z'$ onto $Y \smallsetminus Z$. We say that 
$i \colon Z \to Y$ is an \emph{infinitesimal thickening} if $i$ is a
closed immersion and the corresponding quasi-coherent ideal
$\mathscr{I} \subset \mathscr{O}_Y$ is nilpotent. The $\cdh$-topology
on $\Sch/X$ is defined to be the smallest Grothendieck topology such
that for every abstract blow-up square and every elementary Nisnevich
square, the family of morphisms
$$\{ p \colon Y' \to Y, i \colon Z \to Y \}$$
is a covering of $Y$. In particular, the closed covering of the scheme
$Y$ by its irreducible components is a $\cdh$-covering
as is the closed immersion $Y^{\red} \to Y$.

For the purpose of this paper, we define a \emph{pro-spectrum} to be a
functor from the partially ordered set of positive integers viewed as
a category with a single morphism from $m$ to $n$, if $n \leqslant m$,
to the category of fibrant symmetric spectra, and we define a strict
map of pro-spectra to be a natural transformation. We define the
strict map of pro-spectra $f \colon \{ X^n \} \to \{ Y^n \}$ to be a
\emph{weak equivalence} if for every integer $q$, the induced map of
homotopy groups
$$f_* \colon \{ \pi_q(X^n) \} \to \{ \pi_q(Y^n) \}$$
is an isomorphism of pro-abelian groups. This, we recall, means that
for every $n$, there exists $m \geqslant n$, such that the maps
induced by the structure maps from the kernel and cokernel of the map
$f_*^m \colon \pi_q(X^m) \to \pi_q(Y^m)$ to the kernel and cokernel,
respectively, of the map $f_*^n \colon \pi_q(X^n) \to \pi_q(Y^n)$ are
both zero. We say that the square diagram of strict maps of
pro-spectra 
$$\xymatrix{
{ \{ X^n \} } \ar[r] \ar[d] &
{ \{ Y^n \} } \ar[d] \cr
{ \{ Z^n \} } \ar[r] &
{ \{ W^n \} } \cr
}$$
is \emph{homotopy cartesian} if the canonical map
$$\{ X^n \} \to \{ \holim( Y^n \to W^n \leftarrow Z^n ) \}$$
is a weak equivalence. 

The following result generalizes~\cite[Theorem~6.4]{haesemeyer}
and~\cite[Theorem~3.12]{cortinashaesemeyerschlichtingweibel}.

\begin{theorem}\label{haesemeyertheorem}Let $k$ be an
infinite perfect field such that strong resolution of singularities
holds over $k$, and let $\{ F^n(-) \}$ be a presheaf of pro-spectra on
the category of schemes essentially of finite type over
$k$. Assume that $\{ F^n(-) \}$ takes infinitesimal thickenings to weak
equivalences and finite abstract blow-up squares to homotopy cartesian
squares. Assume further that each $F^n(-)$ takes elementary Nisnevich
squares and squares associated with blow-ups along regular embeddings
to homotopy cartesian squares. Then the canonical map defines a weak
equivalence
$$\{ F^n(X) \} \xrightarrow{\sim}
\{ \mathbb{H}_{\cdh}^{\boldsymbol{\cdot}}(X,F^n(-)) \}$$
of pro-spectra for every scheme $X$ essentially of finite type over $k$.
\end{theorem}

\begin{proof}The proof in outline is analogous to the proof
of~\cite[Theorem~6.4]{haesemeyer}. But some extra care is
needed, since for a map $\{ A^n(-) \} \to \{ B^n(-) \}$ of
pro-objects in the category of sheaves of abelian groups on the
category of schemes essentially of finite type over $k$ to be an
isomorphism, it does not suffice to show that for every such scheme
$X$ and every point $x \in X$, the map $\{ A^n(-)_{X,x} \} \to \{
B^n(-)_{X,x} \}$ of the pro-abelian groups of stalks is an
isomorphism. Instead, one must show that for every scheme $X$ 
essentially of finite type over $k$ and every point $x \in X$, there
exists a Zariski open neighborhood $x \in U \subset X$ such that the
map $\{ A^n(U) \} \to \{ B^n(U) \}$ is an isomorphism of pro-abelian
groups. We point out the necessary changes in the proof of loc.~cit.

First, it follows
from~\cite[Corollary~3.9]{cortinashaesemeyerschlichtingweibel} that
for every scheme $X$ smooth over $k$ and positive integer $n$, 
the canonical map defines a weak equivalence of spectra
$$F^n(X) \xrightarrow{\sim} 
\mathbb{H}_{\cdh}^{\boldsymbol{\cdot}}(X,F^n(-)).$$
In particular, the canonical map defines a weak equivalence of
pro-spectra
$$\{ F^n(X) \} \xrightarrow{\sim} 
\{ \mathbb{H}_{\cdh}^{\boldsymbol{\cdot}}(X,F^n(-)) \}.$$
Now, it follows verbatim from the proof
of~\cite[Proposition~3.12]{haesemeyer} that the statement holds for
every scheme $X$ which is a normal crossing scheme over $k$ in the
sense of loc.~cit., Definition~3.10. 

Next, let $X$ be a Cohen-Macaulay scheme over $k$ and let $D \subset
X$ be an integral subscheme along which $X$ is normally flat. Let
$X'$ be the blow-up of $X$ along $D$ and let $D'$ be the
exceptional fiber. We indicate the changes necessary to the proof 
of op.~cit., Theorem~5.7, in order to show that the square of
pro-spectra
$$\xymatrix{
{ \{ F^n(D') \} } &
{ \{ F^n(X') \} } \ar[l] \cr
{ \{ F^n(D) \} } \ar[u] &
{ \{ F^n(X) \} } \ar[l] \ar[u] \cr
}$$
is homotopy cartesian. It follows from loc.~cit., Proposition~5.4,
that after replacing $X$ by a Zariski open neighborhood of a given
point $x \in X$, there exists a reduction $\tilde{D}$ of $D$ in the
sense of loc.~cit., Definition~5.1, such that $\tilde{D}$ is regularly
embedded in $X$. Let $X_{\tilde{D}}$ be the blow-up of $X$ along
$\tilde{D}$ and let $\tilde{D}'$ be the exceptional fiber. We consider
the following diagram of schemes, where every square is cartesian, and
the induced diagram of pro-spectra.
$$\xymatrix{
{ \,D'\, } \ar@{^{(}->}[r] \ar[dd] &
{ \,\tilde{D}''\, } \ar@{^{(}->}[r] \ar[d] &
{ X' } \ar[d] &
{ \{F^n(D')\} } &
{ \{F^n(\tilde{D}'')\} } \ar[l] &
{ \{F^n(X')\} } \ar[l] \cr
{ } &
{ \,\tilde{D}'\, } \ar@{^{(}->}[r] \ar[d] &
{ X_{\tilde{D}} } \ar[d] &
{ } &
{ \{F^n(\tilde{D}')\} } \ar[u] &
{ \{F^n(X_{\tilde{D}})\} } \ar[l] \ar[u] \cr
{ \,D\, } \ar@{^{(}->}[r] &
{ \,\tilde{D}\, } \ar@{^{(}->}[r] &
{ X } &
{ \{F^n(D)\} } \ar[uu] &
{ \{F^n(\tilde{D})\} } \ar[l] \ar[u] &
{ \{F^n(X)\} } \ar[l] \ar[u] \cr
}$$
The lower right-hand square in the left-hand diagram is a blow-up
along a regular embedding, and therefore, the lower right-hand square
in the right-hand diagram is homotopy cartesian by
assumption. Similarly, the upper right-hand square in the left-hand
diagram is a finite abstract blow-up square, and therefore, the upper
right-hand square in the right-hand diagram also is homotopy
cartesian. Moreover, the left-hand horizontal maps in the left-hand
diagram are infinitesimal thickenings, and therefore, the left-hand
horizontal maps in the right-hand diagram are weak equivalences. It
follows that the outer square in the right-hand diagram is homotopy
cartesian as desired.

Now, it follows verbatim from the proof of op.~cit., Theorem~6.1, that
the theorem holds, if $X$ is a hypersurface in a scheme essentially
smooth over $k$. Similarly, the proof of op.~cit., Corollary~6.2,
shows that the theorem holds, if $X$ is a local complete intersection
in a scheme essentially smooth over $k$.

Finally, in the general case, we proceed as in the proof of op.~cit.,
Theorem~6.4. We may assume that $X$ is integral, since the presheaf
$\{ F^n(-) \}$ is invariant under infinitesimal thickenings and
satisfies descent for finite closed coverings. We now argue by
induction on the dimension $d$ of $X$. The case $d = 0$ follows from
what was proved earlier since $X$ is smooth over $k$, $k$ being
perfect. So we let $d > 0$ and assume the theorem has been proved for
schemes of smaller dimension. Replacing $X$ by an affine open
neighborhood of a given point $x \in X$, we may embed $X$ as a closed
subscheme $X = Z(\mathfrak{p}) \subset U$ of an affine scheme $U =
\Spec A$ essentially smooth over $k$. Since $U$ is regular, there
exists a regular sequence in $\mathfrak{p}$ of length 
$\operatorname{ht}(\mathfrak{p})$. This sequence defines a closed 
subscheme $\tilde{X} \subset U$ which contains $X$ as an irreducible
component and which, after possibly replacing $U$ by a smaller open
neighborhood of $x$, is a local complete intersection. Hence,
the theorem holds for $\tilde{X}$ by what was proved above. Let $X^c
\subset \tilde{X}$ be the union of the components other than $X$. Then
the theorem also holds for the intersection $X \cap X^c$ by
induction. Therefore, it holds for $X$ (and for $X^c$) by the
Mayer-Vietoris sequence associated with the closed covering of
$\tilde{X}$ by $X$ and $X^c$.
\end{proof}

\begin{remark}\label{affinesuffices}In the situation of
Theorem~\ref{haesemeyer}, suppose that $\{ F^n(-) \}$ takes elementary
Nisnevich squares of schemes essentially of finite type over $k$ to
homotopy cartesian squares. Then for every scheme $X$ essentially of
finite type over $k$, the canonical map defines a weak equivalence
$$\{ F^n(X) \} \xrightarrow{\sim} \{
\mathbb{H}_{\Nis}^{\boldsymbol{\cdot}}(X, F^n(-)) \}.$$
Therefore, if $\{ F^n(-) \}$ takes infinitesimal thickenings
(resp.~finite blow-up squares) of affine schemes essentially of finite
type over $k$ to weak equivalences (resp.~homotopy cartesian squares)
then $\{ F^n(-) \}$ takes infinitesimal thickenings (resp.~finite
blow-up squares) of all schemes essentially of finite type over $k$ to
weak equivalences (resp.~homotopy cartesian squares).
\end{remark}

\begin{proof}[Proof of Theorem~\ref{haesemeyer}]We show that the presheaf
of pro-spectra $\{ F^n(-) \}$, where $F^n(X)$ is the mapping fiber of
the cyclotomic trace map
$$\tr \colon K(X) \to \TC^n(X;p),$$
satisfies the hypothesis of Theorem~\ref{haesemeyertheorem}. The functors
$K(-)$ and $\TC^n(-;p)$ both take elementary Nisnevich squares to
homotopy cartesian squares. Indeed, this is proved for $K(-)$
in~\cite[Theorem~10.8]{thomasontrobaugh} and for $\TC^n(-;p)$
in~\cite[Proposition~3.2.1]{gh}. Therefore, the functor $F^n(-)$ takes
elementary Nisnevich squares to homotopy cartesian squares. Next, it
follows from Remark~\ref{affinesuffices} and from~\cite[Theorem~B]{gh5}
that $\{ F^n(-) \}$ takes infinitesimal thickenings to weak equivalences.
Similarly, Remarks~\ref{affinesuffices} and~\cite[Theorem~D]{gh5} show that 
$\{ F^n(-) \}$ takes finite abstract blow-up squares to homotopy
cartesian squares. Finally, two theorems of
Thomason~\cite[Theorem~2.1]{thomason} and
Blumberg and Mandell~\cite[Theorem~1.4]{blumbergmandell} show that the
functors $K(-)$ and $\TC^n(-,;p)$ take squares associated with
blow-ups along regular embeddings to homotopy cartesian
squares. Hence, the same holds for the functor $F^n(-)$. Now,
Theorem~\ref{haesemeyer} follows from Theorem~\ref{haesemeyertheorem}.
\end{proof}

\section{The $\eh$-topology}

Let $X$ be a noetherian scheme. We recall
from~\cite[Definition~2.1]{geisser2} that the $\eh$-topology on the category
$\Sch/X$ of schemes separated and of finite type over $X$ is
defined to be the smallest Grothendieck topology for which both
\'{e}tale coverings and $\cdh$-coverings are coverings. In this
section, we estimate the cohomological dimension of the
$\eh$-topology. Our arguments closely follow 
Suslin-Voevodsky~\cite[Section~12]{suslinvoevodsky}. 

\begin{lemma}\label{reduced}Let $X$ be a noetherian scheme, and let
$F$ be an $\eh$-sheaf of abelian groups on
$\Sch/X$. Then the canonical closed immersion induces an
isomorphism
$$H_{\eh}^*(X,F) \xrightarrow{\sim}
H_{\eh}^*(X^{\red},F).$$
\end{lemma}

\begin{proof}We choose an injective resolution $F \to
I^{\boldsymbol{\cdot}}$ in the category of $\eh$-sheaves
of abelian groups on $\Sch/X$. Since the closed immersion $i \colon
X^{\red} \to X$ is a $\cdh$-covering, we have the equalizer diagram
$$\xymatrix{
{ I^{\boldsymbol{\cdot}}(X) } \ar[r]^(.46){i^*} &
{ I^{\boldsymbol{\cdot}}(X^{\red}) } 
\ar[r]<.7ex>^(.4){\pr_1^*} \ar[r]<-1ex>_(.4){\pr_2^*} &
{ I^{\boldsymbol{\cdot}}(X^{\red} \times_X X^{\red}).
} \cr
}$$
But the closed immersion $i \colon X^{\red} \to X$ is
also a universal homeomorphism so the two projection $\pr_1$ and
$\pr_2$ are equal. Hence, the map $i^*$ is an isomorphism of complexes
of abelian groups. The lemma follows.
\end{proof}

For every morphism $f \colon Y \to X$ between noetherian schemes, we
have the induced functor $f^{-1} \colon \Sch/X \to \Sch/Y$ defined by
$f^{-1}(X'/X) = X'\times_XY/Y$. This functor, in turn, gives rise to
the adjoint pair of functors
$$\xymatrix{
{ (\Sch/X)^{\wedge} } \ar[r]<.7ex>^{f^p} &
{ (\Sch/Y)^{\wedge} } \ar[l]<.7ex>^{f_p} \cr
}$$
where $f_p$ is the restriction along $f^{-1}$ and $f^p$ the left Kan
extension along $f^{-1}$. Since the functor $f_p$ preserves
$\eh$-sheaves, we obtain a morphism of sites
$$f \colon (\Sch/Y)_{\eh} \to (\Sch/X)_{\eh}$$
with the direct image functor $f_*$ given by the restriction of the
functor $f_p$ and with the inverse image functor given by $f^* =
a_{\eh}f^pi_{\eh}$.

We consider two special cases. First, if $f \colon Y \to X$ is
separated and of finite type, the inverse image functor is given by
$f^*F(Y'/Y) = F(Y'/X)$ and has an exact left adjoint functor
$f_{!}$. Hence, in this case, we see that $f^*$ preserves injectives
and that the canonical map defines an isomorphism
$$H_{\eh}^*(Y,F) \xrightarrow{\sim} H_{\eh}^*(Y,f^*F).$$
Second, suppose that $f \colon Y \to X$ is the limit of a cofiltered
diagram $\{X_i\}$ with affine transition maps of schemes separated and
of finite type over $X$. It then follows
from~\cite[Theorem~IV.8.8.2]{EGA} that for every scheme $Y'$ separated
and of finite type over $Y$, there exists an index $i$ and a scheme
$X_i'$ separated and of finite type over $X_i$ such that $Y' =
X_i'\times_{X_i}Y$. Moreover, in this situation, loc.~cit.~implies
that
$$f^pF(Y'/Y) = \colim_{j \to i}F(X_i'\times_{X_i}X_j/X).$$
It is proved in~\cite[Section~12]{suslinvoevodsky} that, in this case,
the functor $f^p$ preserves $\eh$-sheaves, and hence,
that the inverse image functor $f^*$ is given by the same formula.

\begin{prop}\label{limitprop}Let $X$ be a noetherian scheme, and let
$f \colon Y \to X$ be the limit of a cofiltered diagram $\{ X_i \}$
with affine transition maps of schemes separated and of finite type
over $X$. Then for every $\eh$-sheaf $F$ of abelian groups on
$\Sch/X$, the canonical map defines an isomorphism 
$$\colim_i H_{\eh}^*(X_i,F) \xrightarrow{\sim}
H_{\eh}^*(Y,f^*F).$$
\end{prop}

\begin{proof}Let $F \to I^{\boldsymbol{\cdot}}$ be an injective
resolution in the category of $\eh$-sheaves of abelian groups on
$\Sch/X$. Then, since filtered colimits and finite limits of diagrams
of sets commute, the explicit formula for the functor $f^*$ gives an
isomorphism 
$$\colim_i H_{\eh}^*(X_i,F) = \colim_i
H^*\Gamma(X_i,I^{\boldsymbol{\cdot}}) \xrightarrow{\sim}
H^*(\colim_i\Gamma(X_i,I^{\boldsymbol{\cdot}})) =
H^*\Gamma(Y,f^*I^{\boldsymbol{\cdot}}).$$
We claim that for every injective $\eh$-sheaf $I$ of abelian groups
on  $\Sch/X$, the $\eh$-sheaf $f^*I$ of abelian groups on $\Sch/Y$ is
$\Gamma(Y,-)$-acyclic. To show this, it suffices to show that the
\v{C}ech cohomology groups of $\check{H}^*(Y,f^*I)$
vanish~\cite[Proposition~V.4.3]{SGA4}. Since $Y$ is noetherian, every 
$\eh$-covering of $Y$ admits a refinement by a finite
$\eh$-covering. Therefore, it suffices to show that the \v{C}ech
cohomology groups of $f^*I$ with respect to every finite $\eh$-covering
of $Y$ vanish. But this follows immediately from the explicit formula
for the functor $f^*$, since filtered colimits and finite limits of
diagrams of sets commute. This proves the claim. Since $f^*$ is
exact, we conclude that $f^*F \to f^*I^{\boldsymbol{\cdot}}$ is a
resolution by $\Gamma(Y,-)$-acyclic objects in the category of
$\eh$-sheaves of abelian groups on $\Sch/Y$. Therefore, the canonical
map
$$H_{\eh}^*(Y,f^*F) \to H^*\Gamma(Y,f^*I^{\boldsymbol{\cdot}})$$
is an isomorphism. This shows that the map of the statement is an
isomorphism.
\end{proof}

\begin{theorem}\label{ehpcd}Let $X$ be a scheme separated and of
finite type over a separably closed field $k$ of positive
characteristic $p$. Then the $p$-cohomological dimension of $X$ with
respect to the $\eh$-topology is less than or equal to $\dim(X)$.
\end{theorem}

\begin{proof}We must show that for every $\eh$-sheaf $F$ of
$p$-primary torsion abelian groups on $\Sch/X$, the 
cohomology group $H_{\eh}^q(X,F)$ vanishes for $q > d = \dim(X)$. We
follow the proof of~\cite[Theorem~12.5]{suslinvoevodsky} and proceed by
induction on $d$. Suppose first that $d = 0$. By Lemma~\ref{reduced}
we may further assume that $X$ is reduced. It follows
from~\cite[Proposition~2.3]{geisser2} that every $\eh$-covering of $X$
admits a refinement by an \'{e}tale covering of $X$. Therefore, we
conclude from~\cite[Theorem~III.4.1]{SGA4} that the change-of-topology
map defines an isomorphism
$$H_{\et}^q(X,\alpha_*F) \xrightarrow{\sim} 
H_{\eh}^q(X,F).$$
The statement now follows from~\cite[X~Corollary~5.2]{SGA4}.

We next let $d > 0$ and assume that the statement has been proved for
schemes of smaller dimension. By Lemma~\ref{reduced} we may assume
that $X$ is reduced, and by a descending induction on the number of
connected components we may further assume that $X$ is
integral~\cite[Prop~2.3]{geisser2}. We consider the Leray spectral
sequence
$$E_2^{s,t} = H_{\et}^s(X, R^t\alpha_*F)
\Rightarrow H_{\eh}^{s+t}(X, F).$$
Let $x \in X$ be a point of codimension $c < d$, let $\bar{x}$ be a
geometric point lying above $x$, and let $f \colon
\Spec\mathscr{O}_{X,\bar{x}}^{\sh} \to X$ be the
canonical map. It follows from Proposition~\ref{limitprop} that the stalk of
$R^t\alpha_*F$ at $\bar{x}$ is given by
$$(R^t\alpha_*F)_{\bar{x}} 
= H_{\eh}^t(\Spec\mathscr{O}_{X,\bar{x}}^{\sh},f^*F).$$ 
By~\cite[X~Lemma~3.3~(i)]{SGA4} there exists a morphism $g \colon
X' \to X$ from a $c$-dimensional scheme separated and of finite type
over a separably closed field $k'$ and a point $x' \in X'$ lying under
$\bar{x}$ such that $g(x') = x$ and such that the induced map
$$g' \colon \Spec\mathscr{O}_{X',\bar{x}}^{\sh} \to
\Spec\mathscr{O}_{X,\bar{x}}^{\sh}$$
is an isomorphism. Let $f' \colon
\Spec\mathscr{O}_{X',\bar{x}}^{\sh} \to X'$ be the
canonical map. Then 
$$(R^t\alpha_*F)_{\bar{x}} 
= H_{\eh}^t(
\Spec\mathscr{O}_{X',\bar{x}}^{\sh},g'{}^*f^*F)
= H_{\eh}^t(
\Spec\mathscr{O}_{X',\bar{x}}^{\sh},f'{}^*g^*F).$$
Hence, we conclude from Proposition~\ref{limitprop} that
$$(R^t\alpha_*F)_{\bar{x}} = \colim H_{\eh}^t(U',g^*F)$$
where the colimit ranges over all \'{e}tale neighborhoods $U' \to X'$
of $\bar{x}$. Each $U'$ is a $c$-dimensional scheme separated and of
finite type over the separably closed field $k'$. Therefore, the
inductive hypothesis shows that $(R^t\alpha_*F)_{\bar{x}}$ is zero for 
$t > c$. It follows that the sheaf $R^t\alpha_*F$ is supported in
dimension $\max\{0,d-t\}$. We recall from~\cite[X~Corollary~5.2]{SGA4} 
that if $Z$ is a scheme of finite type over a separably closed field
of characteristic $p > 0$, then the \'{e}tale $p$-cohomological
dimension of $Z$ is less than or equal to $\dim(Z)$. This shows that
in the Leray spectral sequence, $E_2^{s,t}$ is zero for $s >
\max\{0,d-t\}$, and hence, the edge-homomorphism 
$$H_{\eh}^q(X, F) \to
H_{\et}^0(X, R^q\alpha_*F)$$
is an isomorphism for $q > d$.

We now fix a cohomology class $h \in H_{\eh}^q(X,F)$ with $q > d$ and
proceed to show that $h$ is zero. There exists an $\eh$-covering $Y
\to X$ such that the restriction of $h$ to $Y$ is zero. Moreover,
by~\cite[Proposition~2.3]{geisser2}, the covering $Y \to X$ admits a
refinement of the form $U' \to X' \to X$, where $U' \to X'$ is an
\'{e}tale covering and $X' \to X$ a proper bi-rational
$\cdh$-covering. We let $X'' \subset X'$ be the closure of the inverse
image of the generic point of $X$, and let $U'' = U \times_{X'}X'' \to
X''$. It follows from~\cite[Lemma~12.4]{suslinvoevodsky} that $X''$ is
a scheme separated and of finite type over $k$ of dimension at most $d$
and that the morphism $X'' \to X$ is proper and bi-rational. We claim
that the restriction $h''$ of the class $h$ to $X''$ is zero. Indeed,
the restriction of $h''$ to $U''$ is zero, and therefore, the image of 
$h''$ by the edge homomorphism of the Leray spectral sequence
$$H_{\eh}^q(X'', F) \to
H_{\et}^0(X'', R^q\alpha_*F)$$
is zero. But we proved above that the edge homomorphism is an
isomorphism, so we find that $h''$ is zero as claimed. To conclude
that $h$ is zero, we choose a proper closed subscheme $Z \subset X$
such that the morphism $X'' \to X$ is an isomorphism outside $Z$ and
define $Z'' = X'' \times_X Z$. Then by~\cite[Proposition~3.2]{geisser2}, we
have a long exact cohomology sequence
$$\cdots \to H_{\eh}^{q-1}(Z'', F) \to 
H_{\eh}^q(X,F) \to
H_{\eh}^q(Z,F) \oplus 
H_{\eh}^q(X'',F) \to \cdots$$
The schemes $Z$ and $Z''$ are of finite type over $k$ and their
dimensions are strictly smaller than $d$. Therefore, by the inductive
hypothesis, the restriction map 
$$H_{\eh}^q(X,F) \to
H_{\eh}^q(X'',F)$$
is an isomorphism for $q > d$. Since the image $h''$ of $h$ by this
map is zero, we conclude that $h$ is zero as stated. This completes
the proof.
\end{proof}

\begin{proof}[Proof of Theorem~\ref{ehpcdtheorem}]We consider the Leray
spectral sequence 
$$E_2^{s,t} = H_{\et}^s(X, R^t\alpha_*F)
\Rightarrow H_{\eh}^{s+t}(X, F).$$
Let $x \in X$ be a point of codimension $c \leqslant d$, and let
$\bar{x}$ be a geometric point lying above $x$. We claim that stalk of
$R^t\alpha_*F$ at $\bar{x}$ vanishes for $t > c$. To prove the claim,
we may assume that $X$ is affine. We write $X$ as a localization
$j \colon X \to X_1$ of a scheme $X_1$ separated and of finite type
over $k$. Then $x_1 = j(x) \subset X_1$ is a point of codimension
$c$. Hence, we find as in the proof of Theorem~\ref{ehpcd}, that the
stalk may be rewritten as a filtered colimit
$$(R^t\alpha_*F)_{\bar{x}} = 
\colim H_{\eh}^t(U',g^*F)$$
where the $U'$ are $c$-dimensional schemes separated and of finite
type over a separably closed field $k'$. The claim now follows
from Theorem~\ref{ehpcd}. We conclude that the sheaf $R^t\alpha_*F$ is
zero for $t > d$, and is supported in dimension $d-t$ for $t
\leqslant d$. Finally, we recall from~\cite[X~Theorem~5.1]{SGA4} that
the \'{e}tale $p$-cohomological dimension of a noetherian $\Fp$-scheme
$Z$ is less than or equal to $\dim(Z) + 1$. Therefore, $E_2^{s,t}$ is
zero for $s+t > d+1$. This completes the proof.
\end{proof}

\section{The de~Rham-Witt sheaves}

We say that the scheme $X$ is essentially smooth over the field $k$
if it can be covered by finitely many affine open subsets of the form
$\smash{ \Spec S^{-1}A }$ with $A$ a smooth $k$-algebra and $S \subset
A$ a multiplicative set. If $X$ is a scheme essentially smooth over a
field $k$, we let $\Sm/X$ be the full subcategory of $\Sch/X$ whose
objects are the schemes smooth over $X$. We define the Zariski,
\'{e}tale, $\cdh$, and $\eh$-topology on $\Sm/X$ to be the
Grothendieck topology induced  by the Zariski, \'{e}tale, $\cdh$, and
$\eh$-topology on $\Sch/X$, respectively, in the sense
of~\cite[III.3.1]{SGA4}. We denote by $\phi_p$ the restriction functor
from the category of presheaves on $\Sch/X$ to the category of
presheaves on $\Sm/X$.

\begin{lemma}\label{smoothenough}Let $k$ be a perfect field of positive
characteristic $p$ such that resolution of singularities holds
over $k$, and let $X$ be a scheme essentially smooth over $k$. Then
for every presheaf $F$ of abelian groups on $\Sch/X$, the canonical
map
$$H_{\tau}^*(X,a_{\tau}\phi_pF) \to 
H_{\tau}^*(X,a_{\tau}F)$$
is an isomorphism for $\tau$ the Zariski, \'{e}tale, $\cdh$, and
$\eh$-topology.
\end{lemma}

\begin{proof}Let $\phi^l$ and $\phi^r$ be the left and right
adjoint functors of $\phi_p$ given by the two Kan extensions. The
functor $\phi_p$ preserves $\tau$-sheaves by the definition of the
induced Grothendieck topology~\cite[III.3.1]{SGA4}. We claim that
also $\phi^r$ preserves $\tau$-sheaves. Indeed, for the
Zariski topology and \'{e}tale topology this follows
from~\cite[Corollary~III.3.3]{SGA4}, and for the $\cdh$-topology and
$\eh$-topology it follows from~\cite[Theorem~III.4.1]{SGA4}. Hence, we
get the diagrams of adjunctions
$$\xymatrix{
{ (\Sch/X)_{\phantom{\tau}}^{\wedge} }
\ar[r]<.7ex>^{a_{\tau}} \ar[d]<.7ex>^{\phi_p} &
{ (\Sch/X)_{\tau}^{\sim} }
\ar[l]<.7ex>^{i_{\tau}} \ar[d]<.7ex>^{\phi_*} &
{ (\Sm/X)_{\phantom{\tau}}^{\wedge} }
\ar[r]<.7ex>^{a_{\tau}} \ar[d]<.7ex>^{\phi^r} &
{ (\Sm/X)_{\tau}^{\sim} }
\ar[l]<.7ex>^{i_{\tau}} \ar[d]<.7ex>^{\phi^!} \cr
{ (\Sm/X)_{\phantom{\tau}}^{\wedge} }
\ar[r]<.7ex>^{a_{\tau}} \ar[u]<.7ex>^{\phi^l} &
{ (\Sm/X)_{\tau}^{\sim} }
\ar[l]<.7ex>^{i_{\tau}} \ar[u]<.7ex>^{\phi^*} &
{ (\Sch/X)_{\phantom{\tau}}^{\wedge} }
\ar[r]<.7ex>^{a_{\tau}} \ar[u]<.7ex>^{\phi_p} &
{ (\Sch/X)_{\tau}^{\sim} }
\ar[l]<.7ex>^{i_{\tau}} \ar[u]<.7ex>^{\phi_*} \cr
}$$
where $\phi_*$ and $\phi^!$ are the restrictions of $\phi_p$ and
$\phi^r$, respectively, to the subcategory of $\tau$-sheaves, and
where $\phi^* = a_{\tau}\phi^li_{\tau}$. It follows that the functor
$\phi_*$ preserves all limits and colimits. Moreover, since the two
diagrams of right adjoint functors commute, the two diagrams of left
adjoint functors commute up to natural isomorphism. In particular,
we have a natural isomorphism
$$a_{\tau}\phi_pF \xrightarrow{\sim} \phi_*a_{\tau}F.$$
Now, the map of the statement is equal to the composition
$$H_{\tau}^*(X,a_{\tau}\phi_pF) \to
H_{\tau}^*(X,\phi_*a_{\tau}F) \to
H_{\tau}^*(X,a_{\tau}F)$$
of the induced isomorphism of cohomology groups and the
edge-homomorphism of the Leray spectral sequence
$$E_2^{s,t} =
H_{\tau}^s(X,R^t\phi_*a_{\tau}F) \Rightarrow
H_{\tau}^{s+t}(X,a_{\tau}F).$$
Since $\phi_*$ is exact the spectral sequence collapses and the
edge-homomorphism is an isomorphism. This completes the proof.
\end{proof}

We let $X$ be a noetherian $\Fp$-scheme and let $\mathscr{O}_X$ be the
presheaf on $\Sch/X$ that to the $X$-scheme $X'$ assigns
the $\Fp$-algebra $\Gamma(X',\mathscr{O}_{X'})$. It is a sheaf for the
\'{e}tale topology, but not for the $\cdh$-topology. We recall the
presheaf
$$W_n\Omega_X^q = W_n\Omega_{\mathscr{O}_X}^q$$
of de~Rham-Witt forms of
Bloch-Deligne-Illusie~\cite[Definition~I.1.4]{illusie}. It follows from  
Proposition~I.1.14 of op.~cit.~that the associated sheaves
$a_{\Zar}W_n\Omega_X^q$ and $a_{\et}W_n\Omega_X^q$ agree and are
quasi-coherent sheaves of $W_n(\mathscr{O}_X)$-modules on $\Sch/X$.

Now, suppose that $X$ is a scheme essentially smooth over a perfect
field $k$ of characteristic $p > 0$. We recall the structure of
the sheaf $a_{\et}W_n\Omega_X^q$ from~\cite{illusie}. We will abuse
notation and write $W_n\Omega_X^q$ for the \'{e}tale sheaf
$a_{\et}W_n\Omega_X^q$ on $\Sch/X$. There is a short exact sequence of
sheaves of abelian groups on $\Sm/X$ for the \'{e}tale topology
\begin{equation}\label{restrictionexactsequence}
0 \to \gr^{n-1}W_n\Omega_X^q \to W_n\Omega_X^q \xrightarrow{R}
W_{n-1}\Omega_X^q \to 0
\end{equation}
where $\gr^{n-1}W_n\Omega_X^q$ is the subsheaf generated by
the images of $V^{n-1} \colon \Omega_X^q \to W_n\Omega_X^q$
and $dV^{n-1} \colon \Omega_X^{q-1} \to W_n\Omega_X^q$. Let
$Z\Omega_X^q$ and $B\Omega_X^{q+1}$ be the kernel and
image sheaves of the differential $d \colon \Omega_X^q \to
\Omega_X^{q+1}$. The inverse Cartier operator
$$C^{-1} \colon \Omega_X^q \to Z\Omega_X^q/B\Omega_X^q$$
is defined as follows. Let $F \colon W_2\Omega_X^q \to \Omega_X^q$ be
the Frobenius map. It satisfies the relations $dF = pFd$, $FV = p$,
and $FdV = d$. The first relation shows that $F$ factors through the
inclusion of the subsheaf $Z\Omega_X^q$ in $\Omega_X^q$, and the two
remaining relations show that the composition $W_2\Omega_X^q \to
Z\Omega_X^q \to Z\Omega_X^q/B\Omega_X^q$ of the Frobenius map and the
canonical projection annihilates the subsheaf
$\gr^1W_2\Omega_X^q$. The inverse Cartier operator is now defined to
be the composition
$$\Omega_X^q \xleftarrow{\sim} W_2\Omega_X^q/\gr^1W_2\Omega_X^q
\xrightarrow{\bar{F}} Z\Omega_X^q/B\Omega_X^q.$$
It is an isomorphism of sheaves of abelian groups on
$\Sm/X$ for the \'{e}tale
topology;~see~\cite[Theorem~7.1]{katz}. The inverse isomorphism
$$C \colon Z\Omega_X^q/B\Omega_X^q \xrightarrow{\sim} \Omega_X^q$$
is the Cartier operator. It gives rise to a chain of subsheaves of
abelian groups
$$0 = B_0\Omega_X^q \subset B_1\Omega_X^1 \subset \cdots \subset
B_s\Omega_X^q \subset \cdots \subset Z_s\Omega_X^q \subset \cdots
Z_1\Omega_X^q \subset Z_0\Omega_X^q = \Omega_X^q$$
where $B_0\Omega_X^q = 0$, $Z_0\Omega_X^q = \Omega_X^q$,
$B_1\Omega_X^q = B\Omega_X^q$, $Z_1\Omega_X^q = Z\Omega_X^q$, and where for
$s \geqslant 2$, $B_s\Omega_X^q$ and $Z_s\Omega_X^q$ are defined to be
the subsheaves of abelian groups of $\Omega_X^q$ characterized by the 
short exact sequences
\begin{equation}\label{definingexactsequences}
\begin{aligned}
{} & 0 \to B_1\Omega_X^q \to B_s\Omega_X^q \xrightarrow{C} B_{s-1}\Omega_X^q \to 0 \cr
{} & 0 \to B_1\Omega_X^q \to Z_s\Omega_X^q \xrightarrow{C} Z_{s-1}\Omega_X^q \to 0 \cr
\end{aligned}
\end{equation}
It then follows from~\cite[Corollary~I.3.9]{illusie} that there is a
short exact sequence
\begin{equation}\label{exactsequenceone}
0 \to
\Omega_X^q/B_{n-1}\Omega_X^q \to
\gr^{n-1}W_n\Omega_X^q \to
\Omega_X^{q-1}/Z_{n-1}\Omega_X^{q-1} \to 0,
\end{equation}
of sheaves of abelian groups on $\Sm/X$ for the
\'{e}tale topology.

\begin{prop}\label{drwprop}Let $k$ be a perfect field of positive
characteristic $p$ and assume that resolution of singularities holds
over $k$. Then for all schemes $X$ essentially smooth over $k$ and
all integers $n \geqslant 1$ and $q \geqslant 0$, the
change-of-topology maps
$$\xymatrix{
{ H_{\Zar}^*(X,a_{\Zar}W_n\Omega_X^q) } \ar[r] \ar[d] &
{ H_{\et}^*(X,a_{\et}W_n\Omega_X^q) } \ar[d] \cr
{ H_{\cdh}^*(X,a_{\cdh}W_n\Omega_X^q) } \ar[r] &
{ H_{\eh}^*(X,a_{\eh}W_n\Omega_X^q) } \cr
}$$
are isomorphisms.
\end{prop}

\begin{proof}Since $a_{\Zar}W_n\Omega_X^q$ and $a_{\et}W_n\Omega_X^q$
are quasi-coherent $W_n(\mathscr{O}_X)$-modules, the top horizontal
map is an isomorphism. We show that the right-hand vertical map is an
isomorphism; the proof for the left-hand vertical map is analogous. By
Lemma~\ref{smoothenough}, it suffices to show that the
change-of-topology map
$$H_{\et}^*(X,a_{\et}\phi_pW_n\Omega_X^q) \to
H_{\eh}^*(X,a_{\eh}\phi_pW_n\Omega_X^q)$$
is an isomorphism. We again abuse notation and write $W_n\Omega_X^q$
for the \'{e}tale sheaf $a_{\et}\phi_pW_n\Omega_X^q$ on $\Sm/X$.

We proceed by induction on $n \geqslant 1$ as
in~\cite[Proposition~6.3]{cortinashaesemeyerschlichtingweibel}. In the case
$n = 1$, we let   
$R_{\et}\Gamma(-, \Omega_X^q)$ be the presheaf of chain complexes
given by a functorial model for the total right derived functor for
the \'{e}tale topology of the functor $\Gamma(-,\Omega_X^q)$ from
$\Sm/X$ to the category of abelian groups. We must show
that the presheaf $R_{\et}\Gamma(-,\Omega_X^q)$ satisfies
descent for the $\eh$-topology. To prove this, it
suffices  by~\cite[Corollary~3.9]{cortinashaesemeyerschlichtingweibel} to
show that for every blow-up square of smooth $X$-schemes
$$\xymatrix{
{ Z' } \ar[r] \ar[d] &
{ Y' } \ar[d] \cr
{ Z } \ar[r] &
{ Y, } \cr
}$$
the induced square of complexes of abelian groups
$$\xymatrix{
{ R_{\et}\Gamma(Z',\Omega_X^q) } &
{ R_{\et}\Gamma(Y',\Omega_X^q) } \ar[l] \cr
{ R_{\et}\Gamma(Z,\Omega_X^q) } \ar[u] &
{ R_{\et}\Gamma(Y,\Omega_X^q) } \ar[l] \ar[u] \cr
}$$
is homotopy cartesian. But this is proved
in~\cite[Chapter~IV,~Theorem~1.2.1]{gros}. 

We next assume the statement for $n-1$ and prove it for $n$. By a
five-lemma argument based on the short exact sequence of
sheaves~(\ref{restrictionexactsequence}), it suffices to show that the
change-of-topology map
$$H_{\et}^*(X,\gr^{n-1}W_n\Omega_X^q) \to
H_{\eh}^*(X,\alpha^*\gr^{n-1}W_n\Omega_X^q)$$
is an isomorphism. Furthermore, by a five-lemma argument based on the
exact sequence~(\ref{exactsequenceone}), it suffices to show that for
all $s \geqslant 0$, the change-of-topology maps
\begin{equation}\label{BZchangeoftopologymaps}
\begin{aligned}
H_{\et}^*(X, B_s\Omega_X^q) & \to
H_{\eh}^*(X, \alpha^*B_s\Omega_X^q) \cr
H_{\et}^*(X, Z_s\Omega_X^q) & \to
H_{\eh}^*(X, \alpha^*Z_s\Omega_X^q) \cr
\end{aligned}
\end{equation}
are isomorphisms. The case $s = 0$ was proved above. To prove the case
$s = 1$, we argue by induction on $q \geqslant 0$: The basic case
$q = 0$ has already been proved, since $B_1\Omega_X^0 = 0$ and $C \colon
Z_1\Omega_X^0 = Z_1\Omega_X^0/B_1\Omega_X^0 \xrightarrow{\sim} \Omega_X^0$, and the
induction step follows by a five-lemma argument based on the exact
sequences of sheaves
$$\begin{aligned}
0 \to Z_1\Omega_X^{q-1} \to \Omega_X^{q-1} & \xrightarrow{d} B_1\Omega_X^q \to
0 \cr
0 \to B_1 \Omega_X^q \to Z_1\Omega_X^q & \xrightarrow{C} \Omega_X^q \to 0. \cr
\end{aligned}$$
Finally, we let $s \geqslant 2$ and assume, inductively, that the
maps~(\ref{BZchangeoftopologymaps}) have been proved to be
isomorphisms for $s-1$. Then a five-lemma argument based on the
short exact sequence of sheaves~(\ref{definingexactsequences}) show
that the maps~(\ref{BZchangeoftopologymaps}) are isomorphisms for
$s$. This completes the proof.
\end{proof}

\begin{lemma}\label{drwcontinuity}Let $X$ be a noetherian $\Fp$-scheme
and let $x \in X$ be a point. Then the canonical map $f \colon
\Spec\mathscr{O}_{X,x} \to X$ induces an isomorphism
$$f^*a_{\Zar}W_n\Omega_X^q \to 
a_{\Zar}W_n\Omega_{\Spec\mathscr{O}_{X,x}}^q$$
of sheaves of abelian groups on $\Sch/\Spec\mathscr{O}_{X,x}$ for the
Zariski topology.
\end{lemma}

\begin{proof}Let $\mathscr{X}$ be a topos. We recall
from~\cite[Definition~I.1.4]{illusie} that, by definition, the de~Rham-Witt
complex is the left adjoint of the forgetful functor $g$ from the
category of $V$-pro-complexes in $\mathscr{X}$ to the category of
$\Fp$-algebras in $\mathscr{X}$. We let $x \in U \subset X$ be an open
neighborhood, and let $i_U$ be the corresponding point of the topos
$(\Sch/X)^{\wedge}$ given by $i_U^*(F) = \Gamma(U,F)$ and 
$i_{U*}(E)(X') = E^{\operatorname{Hom}_X(U,X')}$. Since $g \circ i_{U*} =
i_{U*}\circ g$, we conclude that there is a natural isomorphism
$$\Gamma(U,W_n\Omega_X^q) = \Gamma(U,W_n\Omega_{\mathscr{O}_X}^q)
\xrightarrow{\sim} W_n\Omega_{\Gamma(U,\mathscr{O}_X)}^q.$$
Now, let $Y = \Spec\mathscr{O}_{X,x}$. To prove the lemma, it suffices
to show that for every scheme $Y'$ affine and of finite type over
$Y$, the canonical map
$$\Gamma(Y',f^pW_n\Omega_X^q) \to \Gamma(Y',W_n\Omega_Y^q)$$
is an isomorphism. The discussion preceeding Proposition~\ref{limitprop}
shows that there exists an affine open neighborhood $x \in U \subset
X$ and a scheme $U'$ affine and of finite type over $U$ with $Y' = Y
\times_UU'$ such that the map in question is the canonical map
$$\colim_{x \in V \subset U} \Gamma(V\times_UU',W_n\Omega_X^q)
\to \Gamma(Y\times_UU',W_n\Omega_Y^q).$$
Here the colimit ranges over the affine open neighborhoods
$x \in V \subset U$. Now, by what was proved above, we may identify
this map with the canonical map
$$\colim_{x \in V \subset U} W_n\Omega_{\Gamma(V\times_UU',\mathscr{O}_X)}^q
\to W_n\Omega_{\Gamma(Y\times_UU',\mathscr{O}_Y)}^q.$$
Here the left-hand side is the colimit in the category of
sets. However, since the index category for the colimit is filtered,
the left-hand side is also equal to the colimit in the category of
$V$-pro-complexes in the category of sets. Therefore, the canonical map
in question is an isomorphism. Indeed, being a left adjoint, the
de~Rham-Witt complex preserves colimits, and the canonical map of
$\Fp$-algebras 
$$\colim_{x \in V \subset U}\Gamma(V \times_UU',\mathscr{O}_X) \to
\Gamma(Y\times_UU',\mathscr{O}_Y)$$
is an isomorphism.
\end{proof}

\begin{theorem}\label{drwsheavestheorem}Let $k$ be a perfect field of 
characteristic $p > 0$ such that resolution of singularities holds
over $k$, and let $X$ be a $d$-dimensional scheme essentially of
finite type over $k$. Then the change-of-topology map
$$H_{\et}^i(X,a_{\et}W_n\Omega_X^q) \to
H_{\eh}^i(X,a_{\eh}W_n\Omega_X^q)$$
is a surjection if $i = d$, and both groups vanish if $i > d$. 
\end{theorem}

\begin{proof}We follow the proof
of~\cite[Theorem~6.1]{cortinashaesemeyerschlichtingweibel}. Since the
sheaves $a_{\Zar}W_n\Omega_X^q$ and $a_{\et}W_n\Omega_X^q$ are
quasi-coherent $W_n(\mathscr{O}_X)$-modules, the change-of-topology
map defines an isomorphism
$$H_{\Zar}^i(X,a_{\Zar}W_n\Omega_X^q) \xrightarrow{\sim}
H_{\et}^i(X,a_{\et}W_n\Omega_X^q).$$
In particular, the common group vanishes for $i > d$. We consider
the following commutative diagram, where the horizontal maps are the
change-of-topology maps, and where the vertical maps are induced from
the canonical closed immersion.
$$\xymatrix{
{ H_{\Zar}^i(X,a_{\Zar}W_n\Omega_X^q) } \ar[r] \ar[d] &
{ H_{\eh}^i(X,a_{\eh}W_n\Omega_X^q) } \ar[d] \cr
{ H_{\Zar}^i(X^{\red},a_{\Zar}W_n\Omega_{X^{\red}}^q) }
\ar[r] &
{ H_{\eh}^i(X^{\red},a_{\eh}W_n\Omega_{X^{\red}}^q). }\cr
}$$
By Lemma~\ref{reduced}, the right-hand vertical map is an isomorphism
for all $i \geqslant 0$. Moreover, the cohomology groups on the
left-hand side may be calculated on the small Zariski sites and $X$
and $X^{\red}$, respectively, and the left-hand vertical map may be
identified with the map of sheaf cohomology groups of the small
Zariski site of $X$ induced by the surjective map of sheaves
$$a_{\Zar}W_n\Omega_X^q \to i_*a_{\Zar}W_n\Omega_{X^{\red}}^q$$
induced by the closed immersion $i \colon X^{\red} \to X$. It follows
that the left-hand vertical map is a surjection for $i \geqslant
d$. Therefore, we may assume that $X$ is reduced. 

We proceed by induction on $d$. The case $d = 0$ follows from
Proposition~\ref{drwprop}, since every reduced $0$-dimensional scheme
$X$ of finite type over the perfect field $k$ is smooth over $k$. So
we let $d > 0$ and assume that the statement for schemes of smaller
dimension. We must show that the statement holds for every reduced
$d$-dimensional scheme $X$ essentially of finite type over
$k$. Suppose first that $X$ is affine, and hence, separated. By
resolution of singularities, there exists an abstract blow-up square
$$\xymatrix{
{ Y' } \ar[r]^{i'} \ar[d]^{p'} &
{ X' } \ar[d]^{p} \cr
{ Y } \ar[r]^{i} &
{ X } \cr
}$$
where $X'$ is essentially smooth over $k$ and where the dimensions of
$Y$ and $Y'$ are smaller than $d$. The group
$\smash{ H_{\eh}^i(X',a_{\eh}W_n\Omega_{X'}^q) }$
vanishes for $i > d$ by Proposition~\ref{drwprop} and the groups 
$\smash{ H_{\eh}^i(Y,a_{\eh}W_n\Omega_Y^q) }$ 
and $\smash{ H_{\eh}^i(Y',a_{\eh}W_n\Omega_{Y'}^q) }$
vanish for $i > d-1$ by the induction. Therefore, the Mayer-Vietoris
long exact sequence of $\eh$-cohomology groups associated with the
abstract blow-up square above
$$\cdots \to 
H_{\eh}^{i-1}(Y',a_{\eh}W_n\Omega_{Y'}^q) \to
H_{\eh}^i(X,a_{\eh}W_n\Omega_X^q) \to
\overset{ \displaystyle{
    H_{\eh}^i(X',a_{\eh}W_n\Omega_{X'}^q) }}{ \underset{
    \displaystyle{ H_{\eh}^i(Y,a_{\eh}W_n\Omega_Y^q)
    }}{ \oplus }} \to \cdots$$
shows that the group
$H_{\eh}^i(X,a_{\eh}W_n\Omega_X^q)$ vanishes for $i > d$ as stated. We
must also show that $H_{\eh}^d(X,a_{\eh}W_n\Omega_X^q)$ is zero. We
first show that the common group
$$H_{\Zar}^d(X',a_{\Zar}W_n\Omega_{X'}^q) \xrightarrow{\sim}
H_{\eh}^d(X',a_{\eh}W_n\Omega_{X'}^q)$$
is zero. The left-hand group may be evaluated on the small Zariski
site of $X'$. Now, the theorem of formal
functions~\cite[Corollary~III.4.2.2]{EGA} shows that for every
quasi-coherent $W_n(\mathscr{O}_X')$-module $F$ on the small Zariski
site of $X'$, the higher direct image sheaf $R^dp_*F$ on the small
Zariski site of $X$ is zero. Since $X$ is affine, we conclude from the
Leray spectral sequence that the
$H_{\Zar}^d(X',a_{\Zar}W_n\Omega_{X'}^q)$ is zero as desired. We next
show that the lower horizontal map in the following diagram is
surjective.
$$\xymatrix{
{ H_{\Zar}^{d-1}(X',a_{\Zar}W_n\Omega_{X'}^q) }
\ar[r]^{i'{}^*} \ar[d] &
{ H_{\Zar}^{d-1}(Y',a_{\Zar}W_n\Omega_{Y'}^q) } \ar[d] \cr
{ H_{\eh}^{d-1}(X',a_{\eh}W_n\Omega_{X'}^q) }
\ar[r]^{i'{}^*} &
{ H_{\eh}^{d-1}(Y',a_{\eh}W_n\Omega_{Y'}^q) } \cr
}$$
Here the vertical maps are the change-of-topology maps. By induction,
the right-hand vertical map is surjective, so we may instead show that
the upper horizontal map is surjective. The cohomology groups in the
upper row may be evaluated on small Zariski sites of $X'$ and $Y'$,
respectively. Moreover, the theorem of formal functions shows that the
$R^{d-1}p_*$ is a right-exact functor from the category of
quasi-coherent $W_n(\mathscr{O}_{X'})$-modules to the category of
quasi-coherent $W_n(\mathscr{O}_X)$-modules. Since the closed
immersion $i'$ gives rise to a surjection
$$W_n\Omega_{X'}^q \to i_*'W_n\Omega_{Y'}^q$$
of quasi-coherent $W_n(\mathscr{O}_{X'})$-modules on the small Zariski
site of $X'$, we conclude that the induced map
$$R^{d-1}p_*W_n\Omega_{X'}^q \to R^{d-1}p_*i_*'W_n\Omega_{Y'}^q$$
is a surjection of $W_n(\mathscr{O}_X)$-modules on the small Zariski
site of $X$. As $X$ is assumed to be affine, the Leray  spectral
sequence shows that
$$i'{}^* \colon H_{\Zar}^{d-1}(X',a_{\Zar}W_n\Omega_{X'}^q) \to
H_{\Zar}^{d-1}(Y',a_{\Zar}W_n\Omega_{Y'}^q)$$
is surjective as desired. We conclude from the Mayer-Vietoris exact
sequence that the group $H_{\eh}^d(X,a_{\eh}W_n\Omega_X^q)$ is
zero. This proves the statement of the theorem for $X$ a
$d$-dimensional reduced affine scheme essentially of finite type over
$k$.

It remains to prove the statement for $X$ a general $d$-dimensional
reduced scheme essentially of finite type over $k$. To this end, we
let $\epsilon \colon (\Sch/X)_{\eh} \to (\Sch/X)_{\Zar}$ be the
canonical morphism of sites and consider the Leray spectral sequence
$$E_2^{s,t} = H_{\Zar}^s(X,R^t\epsilon_*a_{\eh}W_n\Omega_X^q) 
\Rightarrow H_{\eh}^{s+t}(X,a_{\eh}W_n\Omega_X^q).$$
Let $x \in X$ be a point of codimension $c$. Then
Proposition~\ref{limitprop} and Lemma~\ref{drwcontinuity} show that the
stalk of $R^t\epsilon_*a_{\eh}W_n\Omega_X^q$ at $x$ is given by 
$$(R^t\epsilon_*a_{\eh}W_n\Omega_X^q)_x^{\phantom{j}} =
H_{\eh}^t(\Spec\mathscr{O}_{X,x},a_{\eh}W_n\Omega_{\Spec\mathscr{O}_{X,x}}^q).$$
We have proved that this group vanishes if either $c > 0$
and $t \geqslant c$ or $c = 0$ and $t > 0$, or equivalently, if $t >
0$ and $c \leqslant t$. It follows that for $t > 0$, the higher
direct image sheaf $\smash{ R^t\epsilon_*a_{\eh}W_n\Omega_X^q }$ is
supported in dimension $<d-t$. Hence, $E_2^{s,t}$ vanishes if 
$t > 0$ and $s+t \geqslant d$. This shows that
$H_{\eh}^i(X,a_{\eh}W_n\Omega_X^q)$ is zero for $i > d$ and that the
edge homomorphism defines a surjection
$$H_{\Zar}^d(X,\epsilon_*a_{\eh}W_n\Omega_X^q) \twoheadrightarrow
H_{\eh}^d(X,a_{\eh}W_n\Omega_X^q).$$
Finally, it follows from Proposition~\ref{drwprop} that the cokernel
of the unit map
$$a_{\Zar}W_n\Omega_X^q \to \epsilon_*a_{\eh}W_n\Omega_X^q$$
is supported on the singular set of $X$ which has dimension strictly
less that $d$. Since the functor $H_{\Zar}^d(X,-)$ is
right-exact, we conclude that the induced map
$$H_{\Zar}^d(X,a_{\Zar}W_n\Omega_X^q) \to
H_{\Zar}^d(X,\epsilon_*a_{\eh}W_n\Omega_X^q)$$
is surjective. This proves the induction step and the theorem.
\end{proof}

\section{The sheaves $a_{\eh}\TR_q^n(-;p)$ and $a_{\eh}\TC_q^n(-;p)$}

Let $X$ be a noetherian $\Fp$-scheme. We briefly recall the
presheaves of fibrant symmetric spectra $\TR^n(-;p)$ and $\TC^n(-;p)$
on $\Sch/X$ and refer to~\cite{gh}
and~\cite{blumbergmandell} for a detailed discussion. Topological
Hochschild homology gives a presheaf $\THH(-)$ of fibrant symmetric
spectra with an action by the multiplicative group $\mathbb{T}$ of
complex numbers of modulus $1$. We let $C_{p^{n-1}} \subset
\mathbb{T}$ be the subgroup of order $p^{n-1}$ and define 
$$\TR^n(-;p) = \THH(-)^{C_{p^{n-1}}}$$
to be the presheaf of fibrant symmetric spectra given by the
$C_{p^{n-1}}$-fixed points. There are two maps of presheaves of
fibrant symmetric spectra
$$R, F \colon \TR^n(-;p) \to \TR^{n-1}(-;p)$$
called the restriction map and the Frobenius map, respectively, and
the presheaf of fibrant symmetric spectra $\TC^n(-;p)$ is defined to
be their homotopy equalizer. It follows immediately from the
definition that the associated presheaves of homotopy groups are
related by a long exact sequence
$$\xymatrix{
{ \cdots } \ar[r]<.2ex> &
{ \TC_q^n(X;p) } \ar[r]<.2ex> &
{ \TR_q^n(X;p) } \ar[r]<.2ex>^(.47){R-F} &
{ \TR_q^{n-1}(X;p) } \ar[r]<.2ex> &
{ \cdots } \cr
}$$
Moreover, by~\cite[Proposition~6.2.4]{hm3}, the sheaves
$a_{\Zar}\TR_q^n(-;p)$ and $a_{\et}\TR_q^n(-;p)$ agree and are
quasi-coherent $W_n(\mathscr{O}_X)$-modules on $\Sch/X$.

\begin{prop}\label{trsmoothisomorphism}Let $k$ be a perfect field of
positive characteristic $p$ such that resolution of singularities
holds over $k$, and let $X$ be a scheme essentially smooth over
$k$. Then for all integers $q$ and $n \geqslant 1$, the
change-of-topology maps
$$\xymatrix{
{ H_{\Zar}^*(X,a_{\Zar}\TR_q^n(-;p)) } \ar[r] \ar[d] &
{ H_{\et}^*(X,a_{\et}\TR_q^n(-;p)) } \ar[d] \cr
{ H_{\cdh}^*(X,a_{\cdh}\TR_q^n(-;p)) } \ar[r] &
{ H_{\eh}^*(X,a_{\eh}\TR_q^n(-;p)) } \cr
}$$
are isomorphisms.
\end{prop}

\begin{proof}The sheaves $a_{\Zar}\TR_q^n(-;p)$ and
$a_{\et}\TR_q^n(-;p)$ are sheaves of quasi-coherent
$W_n(\mathscr{O}_X)$-modules. Therefore, the top horizontal map is an
isomorphism. We show that the right-hand vertical map is an
isomorphism; the proof for the left-hand vertical map is analogous. 
It suffices by Lemma~\ref{smoothenough}, to show that the
change-of-topology map
$$H_{\et}^*(X,a_{\et}\phi_p\TR_q^n(-;p)) \to
H_{\eh}^*(X,a_{\eh}\phi_p\TR_q^n(-;p))$$
is an isomorphism. We recall from~\cite[Theorem~B]{h} that there is a
canonical isomorphism of sheaves of abelian groups on
$\Sm/X$ for the Zariski topology
$$\bigoplus_{m \geqslant 0} a_{\Zar}\phi_pW_n\Omega_X^{q-2m} \xrightarrow{\sim}
a_{\Zar}\phi_p\TR_q^n(-;p).$$
It induces an isomorphism of associated sheaves for the \'{e}tale
topology and for the $\eh$-topology. The proposition now follows from 
Proposition~\ref{drwprop} above.
\end{proof}

\begin{cor}\label{trcohomology}Let $k$ be a perfect field of positive
characteristic $p$ and assume that resolution of singularities holds
over $k$. Let $X$ be a $d$-dimensional scheme essentially of finite
type over $k$. Then for all integers $q$ and $n \geqslant 1$,
$$H_{\eh}^{d+1}(X,a_{\eh}\TR_q^n(-;p)) = 0.$$
\end{cor}

\begin{proof}The proof is by induction on $d$ and is analogous to the
first part of the proof of Theorem~\ref{drwsheavestheorem} above.
\end{proof}

Let $k$ be a perfect field of characteristic $p > 0$ and let $X$
be a scheme essentially of finite type over $k$. We consider the long
exact sequence
$$\xymatrix{
{ \cdots } \ar[r] &
{ \{ a_{\et}\TC_q^n(-;p) \} } \ar[r] &
{ \{ a_{\et}\TR_q^n(-;p) \} } \ar[r]^{\id-F} &
{ \{ a_{\et}\TR_q^n(-;p) \} } \ar[r] &
{ \cdots } \cr
}$$
of \'{e}tale sheaves of pro-abelian groups on $\Sch/X$. Here the
structure maps in the pro-abelian groups are the restriction maps
$R$. We recall from~\cite{h8} that there is a canonical map compatible
with restriction and Frobenius operators
\begin{equation}\label{derhamwitttr}
a_{\et}W_n\Omega_X^q \to a_{\et}\TR_q^n(-;p)
\end{equation}
and that this map is an isomorphism, for $q \leqslant 1$. We remark
that the assumption in op.~cit.~that $p$ be odd is
unnecessary. Indeed,~\cite[Theorem~4.3]{costeanu} shows that the result
is valid also for $p = 2$. We examine the map $\id - F$ in degrees
$q \leqslant 1$. 

\begin{lemma}\label{wittfrobenius}Let $X$ be a noetherian scheme over
$\Fp$. Then there is an exact sequence of sheaves of pro-abelian
groups on $\Sch/X$ for the \'{e}tale topology:
$$\xymatrix{
{ 0 } \ar[r] &
{ \{ \Z/p^n\Z \} } \ar[r] &
{ \{ a_{\et}W_n(\mathscr{O}_X) \} } \ar[r]^{\id - F} &
{ \{ a_{\et}W_n(\mathscr{O}_X) \} } \ar[r] &
{ 0. } \cr
}$$
\end{lemma}

\begin{proof}Since $X$ is a scheme over $\Fp$, the Frobenius map $F$ 
agrees with the map $\{ W_n(\varphi) \}$ induced by the Frobenius
endomorphism of $X$. We prove that for every strictly henselian
noetherian local $\Fp$-algebra $(A,\mathfrak{m},\kappa)$, the
sequence
$$\xy
(0,0)*{W_n(A)};
(22,0)*{W_n(A)};
{\ar (16,0)*{};(6,0)*{};};
(-22,0)*{\Z/p^n\Z};
{ \ar (-6,0)*{};(-16,0)*{};};
(11,3)*{\scriptstyle{\id - W_n(\varphi)}};
(40,0)*{0};
{\ar (38,0)*{};(28,0)*{};};
(-40,0)*{0};
{ \ar (-28,0)*{};(-38,0)*{};};
\endxy$$
is exact. Since $A$ is strictly henselian, the map $\id - \varphi
\colon A \to A$ is surjective. An induction argument based on the
diagram with exact rows
$$\xymatrix{
{ 0 } \ar[r] &
{ A } \ar[r]^(.4){V^{n-1}} \ar[d]^{\id - \varphi} &
{ W_n(A) } \ar[r]^(.45){R} \ar[d]^{\id - W_n(\varphi)} &
{ W_{n-1}(A) } \ar[r] \ar[d]^{\id - W_{n-1}(\varphi)} &
{ 0 } \cr
{ 0 } \ar[r] &
{ A } \ar[r]^(.4){V^{n-1}} &
{ W_n(A) } \ar[r]^(.45){R} &
{ W_{n-1}(A) } \ar[r] &
{ 0 } \cr
}$$
shows that the map $\id - F$ is surjective as stated. To identify the
kernel of $\id - F$, we consider the following diagram
$$\xymatrix{
{ 0 } \ar[r] &
{ W_n(\mathfrak{m}) } \ar[r] \ar[d]^{\id - W_n(\varphi)} &
{ W_n(A) } \ar[r] \ar[d]^{\id - W_n(\varphi) } &
{ W_n(\kappa) } \ar[r] \ar[d]^{\id - W_n(\varphi) } &
{ 0 } \cr
{ 0 } \ar[r] &
{ W_n(\mathfrak{m}) } \ar[r] &
{ W_n(A) } \ar[r] &
{ W_n(\kappa) } \ar[r] &
{ 0 } \cr
}$$
We wish to show that the unit map $\eta \colon W_n(\Fp) \to W_n(A)$ is
an isomorphism onto the kernel of the middle vertical map. We have
$\\operatorname{im}(\eta) \subset \ker(\id - W_n(\varphi))$. Moreover, since $\kappa$
is a domain, the composition
$$W_n(\Fp) \to W_n(A) \to W_n(\kappa)$$
of $\eta$ and the canonical projection is an isomorphism of $W_n(\Fp)$
onto the kernel of the right-hand vertical map in the diagram
above. Hence, it will suffice to show that, in the diagram above, the
left-hand vertical map is injective. Moreover, an induction argument
based on the diagram
$$\xymatrix{
{ 0 } \ar[r] &
{ W_{n-1}(\mathfrak{m}) } \ar[r]^(.52){V} \ar[d]^{\id - W_{n-1}(\varphi)} &
{ W_n(\mathfrak{m}) } \ar[r]^(.6){R^{n-1}} \ar[d]^{\id - W_n(\varphi)} &
{ \mathfrak{m} } \ar[r] \ar[d]^{\id - \varphi} &
{ 0 } \cr
{ 0 } \ar[r] &
{ W_{n-1}(\mathfrak{m}) } \ar[r]^(.52){V} &
{ W_n(\mathfrak{m}) } \ar[r]^(.6){R^{n-1}} &
{ \mathfrak{m} } \ar[r] &
{ 0 } \cr
}$$
shows that it suffices to consider the case $n = 1$. In this case, we
recall that by a theorem of Krull~\cite[Corollary~0.7.3.6]{EGA}, the
$\mathfrak{m}$-adic topology on $A$ is separated. It follows that the 
map $\id - \varphi \colon \mathfrak{m} \to \mathfrak{m}$ is injective
as desired.
\end{proof}

\begin{lemma}\label{derhamwittfrobenius}Let $X$ be a noetherian scheme
over $\Fp$. Then the map of sheaves of pro-abelian groups on $\Sch/X$
for the \'{e}tale topology
$$\id - F \colon \{ a_{\et}W_n\Omega_X^1 \} \to
\{ a_{\et}W_n\Omega_X^1 \}$$
is an epimorphism.
\end{lemma}

\begin{proof}It will suffice to show that for every strictly henselian
noetherian local $\Fp$-algebra $(A,\mathfrak{m},\kappa)$, the map
$$R - F \colon W_n\Omega_A^1 \to W_{n-1}\Omega_A^1$$
is surjective. Since $A$ is local, the abelian group $W_n\Omega_A^1$ is
generated by elements of the form $ad\log[x]_n$ with $a \in W_n(A)$ and
$x \in 1+\mathfrak{m}$. Moreover, we have
$$(R-F)(ad\log[x]_n) = (R-F)(a)d\log[x]_{n-1}.$$
Hence, the proof of Lemma~\ref{wittfrobenius} shows that $R - F$ is an
epimorphism. The lemma follows.
\end{proof}

\begin{prop}\label{tcvanishing}Let $X$ be a noetherian scheme over $\Fp$. Then:
\begin{enumerate}
\item[(i)] The sheaf of pro-abelian groups
$\{ a_{\et}\TC_q^n(-;p) \}$ is zero for $q < 0$. 
\item[(ii)]The sheaf of pro-abelian groups
$\{ a_{\et}\TC_0^n(-;p) \}$ is canonically isomorphic to the sheaf
of pro-abelian groups $\{ \Z/p^n\Z \}$. 
\item[(iii)] There is a long exact sequence of sheaves of pro-abelian
groups
$$\xymatrix{
{ \cdots } \ar[r] &
{ \{ a_{\et}\TC_1^n(-;p) \} } \ar[r] &
{ \{ a_{\et}\TR_1^n(-;p) \} } \ar[r]^{\id-F} &
{ \{ a_{\et}\TR_1^n(-;p) \} } \ar[r] &
{ 0 } \cr
}$$
\end{enumerate}
\end{prop}

\begin{proof}This follows immediately from the fact that the
map~(\ref{derhamwitttr}) is an isomorphism, for $q \leqslant 1$, and
from Lemmas~\ref{wittfrobenius} and~\ref{derhamwittfrobenius}.
\end{proof}

\begin{question}\label{remark}We do not know whether or not the
sequence of sheaves
$$0 \to \{ a_{\et}\TC_q^n(-;p) \} \to
\{ a_{\et}\TR_q^n(-;p) \} \xrightarrow{1-F}
\{ a_{\et}\TR_q^n(-;p) \} \to 0$$
on $\Sch/X$ for the \'{e}tale topology is exact for $q \geqslant 1$.
\end{question}

\begin{theorem}\label{tcsheavestheorem}Let $k$ be a field of positive
characteristic $p$ and assume that resolution of singularities holds
over $k$. Let $X$ be a $d$-dimensional scheme essentially of finite
type over $k$. Then the map of pro-abelian groups
$$\{ H_{\et}^i(X,a_{\et}\TC_q^n(-;p)) \} \to
\{ H_{\eh}^i(X,a_{\eh}\TC_q^n(-;p)) \}$$
induced by the change-of-topology maps is an epimorphism if $q = 1$
and $i = d+1$, an isomorphism if $q = 0$, and the two pro-abelian
groups are zero if $q \leqslant -1$.
\end{theorem}

\begin{proof}The statement for $q < 0$ follows from
Proposition~\ref{tcvanishing}~(i), and the statement for $q = 0$ follows
from Proposition~\ref{tcvanishing}~(ii) and from~\cite[Theorem~3.6]{geisser2}
which shows that for the constant sheaf $\Z/p^n\Z$, the
change-of-topology map
$$H_{\et}^*(X,\Z/p^n\Z) \to
H_{\eh}^*(X,\Z/p^n\Z)$$
is an isomorphism. To prove the statement for $q = 1$, we fix a
$d$-dimensional scheme $X$ essentially of finite type over $k$. The
$p$-cohomological dimension of $X$ with respect to both the \'{e}tale
topology and the $\eh$-topology is less than or equal to $d+1$; see
Theorem~\ref{ehpcdtheorem}. Let us write
$$\cdots \to \{ F_{-2}^n \} \to \{ F_{-1}^n \} \to \{ F_0^n \} \to 0$$ 
for the long exact sequence of \'{e}tale sheaves of pro-abelian groups on
$\Sch/X$ from Proposition~\ref{tcvanishing}~(iii). Then we have hypercohomology
spectral sequences
$$\begin{aligned}
E_1^{s,t} & = \{ H_{\et}^t(X,F_{-s}^n) \} \Rightarrow
\{ \mathbb{H}_{\et}^{s+t}(X,F_{\boldsymbol{\cdot}}^n) \}
\cr
E_1^{s,t} & = \{ H_{\eh}^t(X,\alpha^*F_{-s}^n) \} \Rightarrow
\{ \mathbb{H}_{\eh}^{s+t}(X,\alpha^*F_{\boldsymbol{\cdot}}^n) \} \cr
\end{aligned}$$
with the $d_1$-differentials induced by the differential in the
cochain complex $\{ F_{\boldsymbol{\cdot}}^n \}$. Since the complex
$\{ F_{\boldsymbol{\cdot}}^n \}$ is exact and since the cohomological
dimension of $X$ is bounded, the spectral sequences converge and their
abutment is zero. Moreover, the change-of-topology maps induce a map
of spectral sequences from the top spectral sequence to the bottom
spectral sequence. We proved in Corollary~\ref{trcohomology} that the
cohomology groups $\smash{ H_{\eh}^{d+1}(X,\alpha^*F_{-3}^n) }$ and
$\smash{ H_{\eh}^{d+1}(X,\alpha^*F_{-1}^n) }$ vanish. Similarly, the
cohomology groups $\smash{ H_{\et}^{d+1}(X,F_{-3}^n) }$ and $\smash{
H_{\et}^{d+1}(X,F_{-1}^n) }$ vanish, since the sheaves $F_{-3}^n$
and $F_{-1}^n$ are quasi-coherent $W_n(\mathscr{O}_X)$-modules. 
Hence, the map of hypercohomology spectral sequences gives rise to a
commutative diagram
$$\xymatrix{
{ \{ H_{\et}^d(X,F_{-1}^n) \} } \ar[r]^{\id-F} \ar[d] &
{ \{ H_{\et}^d(X,F_0^n) \} } \ar[r] \ar[d] &
{ \{ H_{\et}^{d+1}(X,F_{-2}^n) \} } \ar[r] \ar[d] &
{ 0 } \cr
{ \{ H_{\eh}^d(X,\alpha^*F_{-1}^n) \} } \ar[r]^{\id-F} &
{ \{ H_{\eh}^d(X,\alpha^*F_0^n) \} } \ar[r] &
{ \{ H_{\eh}^{d+1}(X,\alpha^*F_{-2}^n) \} } \ar[r] &
{ 0 } \cr
}$$
with exact rows. Here the middle and left-hand vertical maps are both
equal to the change-of-topology map
$$\{ H_{\et}^d(X,a_{\et}\TR_1^n(-;p)) \} \to
\{ H_{\eh}^d(X,a_{\eh}\TR_1^n(-;p)) \}$$
which by~\cite{h8} coincides with the change-of-topology map
$$\{ H_{\et}^d(X,a_{\et}W_n\Omega_X^1) \} \to
\{ H_{\eh}^d(X,a_{\eh}W_n\Omega_X^1) \}.$$
We proved in Theorem~\ref{drwsheavestheorem} that this map is an
epimorphism. Therefore, also the right-hand vertical map,
$$\{ H_{\et}^{d+1}(X,a_{\et}\TC_1^n(-;p)) \} \to
\{ H_{\eh}^{d+1}(X,a_{\eh}\TC_1^n(-;p)) \},$$
is an epimorphism. This completes the proof.
\end{proof}

\section{Proof of Theorems~\ref{vanishing} and~\ref{tctheorem}}

We first show that, in the statement of Theorem~\ref{tctheorem}, we may
replace the Zariski topology and $\cdh$-topology by the \'{e}tale
topology and $\eh$-topology, respectively.

\begin{theorem}\label{hyperchangeoftopology}Let $k$ be a perfect field
of positive characteristic $p$ and assume that resolution of
singularities holds over $k$. Then for every scheme $X$ essentially of
finite type over $k$, the horizontal maps in the diagram of
change-of-topology maps
$$\xymatrix{
{ \mathbb{H}_{\Zar}^{\boldsymbol{\cdot}}(X,\TC^n(-;p)) }
\ar[r] \ar[d] & 
{ \mathbb{H}_{\et}^{\boldsymbol{\cdot}}(X,\TC^n(-;p)) }
\ar[d] \cr
{ \mathbb{H}_{\cdh}^{\boldsymbol{\cdot}}(X,\TC^n(-;p)) }
\ar[r] &
{ \mathbb{H}_{\eh}^{\boldsymbol{\cdot}}(X,\TC^n(-;p)). }
\cr
}$$
are weak equivalences.
\end{theorem}

\begin{proof}The diagram in the statement agrees with the diagram of 
homotopy equalizers of the maps induced by the restriction and
Frobenius maps from the following diagram of change-of-topology maps
to itself.
$$\xymatrix{
{ \mathbb{H}_{\Zar}^{\boldsymbol{\cdot}}(X,\TR^n(-;p)) }
\ar[r] \ar[d] & 
{ \mathbb{H}_{\et}^{\boldsymbol{\cdot}}(X,\TR^n(-;p)) }
\ar[d] \cr
{ \mathbb{H}_{\cdh}^{\boldsymbol{\cdot}}(X,\TR^n(-;p)) }
\ar[r] &
{ \mathbb{H}_{\eh}^{\boldsymbol{\cdot}}(X,\TR^n(-;p)). }
\cr
}
$$
Hence, it suffices to prove that the horizontal maps in this diagram
are weak equivalences. The top horizontal map induces a map from the
spectral sequence
$$E_{s,t}^2 = H_{\Zar}^{-s}(X,a_{\Zar}\TR_t^n(-;p)) \Rightarrow
\mathbb{H}_{\Zar}^{-s-t}(X,\TR^n(-;p))$$
to the spectral sequence
$$E_{s,t}^2 = H_{\et}^{-s}(X,a_{\et}\TR_t^n(-;p)) \Rightarrow
\mathbb{H}_{\et}^{-s-t}(X,\TR^n(-;p)).$$
The map of $E^2$-terms is given by the change-of-topology map
$$H_{\Zar}^{-s}(X,a_{\Zar}\TR_t^n(-;p)) \to 
H_{\et}^{-s}(X,a_{\et}\TR_t^n(-;p))$$
and is an isomorphism, since  $a_{\Zar}\TR_t^n(-;p)$ and
$a_{\et}\TR_t^n(-;p)$ are quasi-coherent
$W_n(\mathscr{O}_X)$-modules. It follows that the map
$$\mathbb{H}_{\Zar}^{\boldsymbol{\cdot}}(X,\TR^n(-;p))
\to 
\mathbb{H}_{\et}^{\boldsymbol{\cdot}}(X,\TR^n(-;p))$$
is a weak equivalence as stated. It remains to prove that also the
lower horizontal map is a weak equivalence.

Suppose first that $X$ is essentially smooth over $k$. The left-hand
vertical map induces a map from the spectral sequence
$$E_{s,t}^2 = H_{\et}^{-s}(X,a_{\et}\TR_t^n(-;p)) \Rightarrow
\mathbb{H}_{\et}^{-s-t}(X,\TR^n(-;p)).$$
to the spectral sequence
$$E_{s,t}^2 = H_{\eh}^{-s}(X,a_{\eh}\TR_t^n(-;p)) \Rightarrow
\mathbb{H}_{\eh}^{-s-t}(X,\TR^n(-;p)).$$
The map of $E^2$-terms is the change-of-topology map
$$H_{\et}^{-s}(X,a_{\et}\TR_t^n(-;p)) \to 
H_{\eh}^{-s}(X,a_{\eh}\TR_t^n(-;p))$$
which is an isomorphism by Proposition~\ref{trsmoothisomorphism}. Therefore,
the change-of-topology map
$$\mathbb{H}_{\et}^{\boldsymbol{\cdot}}(X,\TR^n(-;p))
\to
\mathbb{H}_{\eh}^{\boldsymbol{\cdot}}(X,\TR^n(-;p))$$
is a weak equivalence. One shows similarly that the change-of-topology
map
$$\mathbb{H}_{\Zar}^{\boldsymbol{\cdot}}(X,\TR^n(-;p))
\to
\mathbb{H}_{\cdh}^{\boldsymbol{\cdot}}(X,\TR^n(-;p))$$
is a weak equivalence. Hence, for $X$ essentially smooth over $k$, 
the lower horizontal map is a weak equivalence.

Finally, we show that for a general scheme $X$ essentially of finite
type over $k$, the lower horizontal map in the diagram at the
beginning of the proof is a weak equivalence. Lemma~\ref{reduced} and 
the appropriate descent spectral sequences show that the vertical maps
in the diagram
$$\xymatrix{
{ \mathbb{H}_{\cdh}^{\boldsymbol{\cdot}}(X,\TR^n(-;p)) }
\ar[r] \ar[d] & 
{ \mathbb{H}_{\eh}^{\boldsymbol{\cdot}}(X,\TR^n(-;p)) }
\ar[d] \cr
{ \mathbb{H}_{\cdh}^{\boldsymbol{\cdot}}(
X^{\red}, \TR^n(-;p)) } \ar[r] &
{ \mathbb{H}_{\eh}^{\boldsymbol{\cdot}}(
X^{\red},\TR^n(-;p)) } \cr
}$$
are weak equivalences. Hence, we may assume that $X$ is reduced. We
proceed by induction on the dimension $d = \dim(X)$. The case $d = 0$
has already been proved since every reduced $0$-dimensional scheme of
finite type over the perfect field $k$ is smooth over $k$. So let $X$
be a reduced $d$-dimensional scheme separated and essentially of finite
type over $k$ and assume that the statement has been proved for all
schemes separated and essentially of finite type over $k$ of dimension
at most $d-1$. We may further assume that $X$ is affine, and hence,
separated. Therefore, by the assumption of resolution of
singularities, there exists an abstract blow-up square
$$\xymatrix{
{ Y' } \ar[r]^{i'} \ar[d]^{p'} &
{ X' } \ar[d]^{p} \cr
{ Y } \ar[r]^{i} &
{ X } \cr
}$$
where is $X'$ essentially smooth over $k$ and where the dimensions of
$Y$ and $Y'$ are strictly smaller than $d$. The change-of-topology map
gives rise to a map from the square of symmetric spectra
$$\xymatrix{
{ \mathbb{H}_{\cdh}^{\boldsymbol{\cdot}}(Y',\TR^n(-;p)) }
&
{ \mathbb{H}_{\cdh}^{\boldsymbol{\cdot}}(X',\TR^n(-;p)) }
\ar[l]_{i'{}^*} \cr
{ \mathbb{H}_{\cdh}^{\boldsymbol{\cdot}}(Y,\TR^n(-;p)) }
\ar[u]_{p'{}^*} &
{ \mathbb{H}_{\cdh}^{\boldsymbol{\cdot}}(X,\TR^n(-;p)) }
\ar[l]_{i^*} \ar[u]_{p^*} \cr
}$$
to the square of symmetric spectra
$$\xymatrix{
{ \mathbb{H}_{\eh}^{\boldsymbol{\cdot}}(Y',\TR^n(-;p)) }
&
{ \mathbb{H}_{\eh}^{\boldsymbol{\cdot}}(X',\TR^n(-;p)) }
\ar[l]_{i'{}^*} \cr
{ \mathbb{H}_{\eh}^{\boldsymbol{\cdot}}(Y,\TR^n(-;p)) }
\ar[u]_{p'{}^*} &
{ \mathbb{H}_{\eh}^{\boldsymbol{\cdot}}(X,\TR^n(-;p)) }
\ar[l]_{i^*} \ar[u]_{p^*} \cr
}$$
both of which are homotopy cartesian. The map of upper right-hand
terms is a weak equivalence, since $X'$ is essentially smooth over
$k$. Moreover, since $Y$ and $Y'$ have dimension at most $d-1$, the
maps of the two left-hand terms are weak equivalences by the inductive
hypothesis. It follows that the map of lower right-hand terms is a
weak equivalence. This completes the proof.
\end{proof}

\begin{proof}[Proof of Theorem~\ref{tctheorem}]Let $X$ be a
$d$-dimensional scheme essentially of finite type over the field
$k$. By Theorem~\ref{hyperchangeoftopology}, it suffices to show that
the change-of-topology map
$$\{ \mathbb{H}_{\et}^{-q}(X,\TC^n(-;p)) \} \to
\{ \mathbb{H}_{\eh}^{-q}(X,\TC^n(-;p)) \}$$
is an isomorphism of pro-abelian groups for $q < -d$, and an
epimorphism of pro-abelian groups for $q = -d$. We recall
from~\cite[X~Theorem~5.1]{SGA4} and Theorem~\ref{ehpcdtheorem} that
the $p$-cohomological dimension of $X$ for the both the \'{e}tale
topology and the $\eh$-topology is at most $d+1$. In
particular, the spectral
sequences
$$\begin{aligned}
E_{s,t}^2 & = H_{\et}^{-s}(X,a_{\et}\TC_t^n(-;p))
\Rightarrow \mathbb{H}_{\et}^{-s-t}(X,\TC^n(-;p)) \cr
E_{s,t}^2 & = H_{\eh}^{-s}(X,a_{\et}\TC_t^n(-;p))
\Rightarrow \mathbb{H}_{\eh}^{-s-t}(X,\TC^n(-;p)) \cr
\end{aligned}$$
converge strongly and the induced filtration of the abutment is of
finite length less than or equal to $d+1$. Therefore, as $n$ varies,
these spectral sequences give rise to strongly convergent spectral
sequences of pro-abelian groups 
$$\begin{aligned}
E_{s,t}^2 & = \{ H_{\et}^{-s}(X,a_{\et}\TC_t^n(-;p)) \}
\Rightarrow \{ \mathbb{H}_{\et}^{-s-t}(X,\TC^n(-;p)) \} \cr
E_{s,t}^2 & = \{ H_{\eh}^{-s}(X,a_{\eh}\TC_t^n(-;p)) \}
\Rightarrow \{ \mathbb{H}_{\eh}^{-s-t}(X,\TC^n(-;p)) \}. \cr
\end{aligned}$$
The map in question induces a map of spectral sequences from the top
spectral sequence to the bottom spectral sequence which, on
$E^2$-terms, is given by the change-of-topology map in sheaf
cohomology. The two $E^2$-terms are concentrated in the region where
$-d-1 \leqslant s \leqslant 0$ and $t \geqslant 0$. Moreover,
Theorem~\ref{tcsheavestheorem} shows that the change-of-topology map is
an isomorphism of pro-abelian groups if $t = 0$, and an epimorphism
of pro-abelian groups if $t = 1$ and $s = -d-1$. The theorem follows.
\end{proof}

\begin{proof}[Proof of Theorem~\ref{vanishing}]Let $X$ be a
$d$-dimensional scheme essentially of finite type over the field
$k$. Then Theorem~\ref{haesemeyer} shows that the diagram of
pro-spectra
$$\xymatrix{
{ K(X) } \ar[r] \ar[d] &
{ \mathbb{H}_{\cdh}^{\boldsymbol{\cdot}}(X,K(-)) }
\ar[d] \cr
{ \{ \TC^n(X;p) \} } \ar[r] &
{ \{
  \mathbb{H}_{\cdh}^{\boldsymbol{\cdot}}(X,\TC^n(-;p)) \} }
}$$
is homotopy cartesian. We conclude from Theorem~\ref{tctheorem} that
the canonical map 
$$K_q(X) \to \mathbb{H}_{\cdh}^{-q}(X,K(-))$$
is an isomorphism, for $q < -d$, and an epimorphism for $q = -d$. The
groups on the right-hand side are the abutment of the spectral
sequence
$$E_{s,t}^2 = H_{\cdh}^s(X,a_{\cdh}K_t(-)) \Rightarrow
\mathbb{H}_{\cdh}^{-s-t}(X,K(-)).$$
The assumption that resolution of singularities holds over $k$ implies
that every $\cdh$-covering of an object in $\Sch/X$ admits a refinement
to a $\cdh$-covering by schemes essentially smooth over $k$. This, in
turn, implies that the sheaf $a_{\cdh}K_t(-)$ vanishes for $t < 0$,
and is canonically isomorphic to the constant sheaf $\Z$ for $t =
0$. We recall from~\cite[Theorem~12.5]{suslinvoevodsky} that the
$\cdh$-cohomological dimension of $X$ is less than or
equal to $d$. Therefore, the groups $E_{s,t}^2$ are zero unless $-d
\leqslant s \leqslant 0$ and $t \geqslant 0$. It follows that $\smash{
  \mathbb{H}_{\cdh}^{-q}(X,K(-)) }$ is zero for $q < -d$, 
and that $\smash{ \mathbb{H}_{\cdh}^d(X,K(-)) }$ is canonically
isomorphic to $\smash{ H_{\cdh}^d(X,\Z) }$. Hence $K_q(X)$ vanishes
for $q < -d$ as stated.
\end{proof}

\begin{remark}The proof above also shows that $K_{-d}(X)$ surjects
onto $H_{\cdh}^d(X,\Z)$.
\end{remark}

\begin{theorem}\label{finitetype}Suppose that strong resolution of
singularities holds over every infinite perfect field of positive
characteristic $p$ and let $X$ be a $d$-dimensional scheme of finite
type over some field of characteristic $p$. Then $K_q(X)$ vanishes for
$q < -d$.
\end{theorem}

\begin{proof}We first let $F$ be a finite field of characteristic
$p$, and let $X$ be a $d$-dimensional scheme essentially of finite
type over $F$. We let $\ell$ be a prime number, let $F'$ be a Galois
extension of $F$ with Galois group isomorphic to the additive group
$\Z_{\ell}$ of $\ell$-adic integers, and let $X'$ be the base-change
of $X$ along $\Spec F' \to \Spec F$. Then $F'$ is an infinite perfect
field. By assumption, strong resolution of singularities holds over
$F'$, so Theorem~\ref{vanishing} shows that $K_q(X')$ vanishes for $q
<  d$. We claim that the kernel of the pull-back map $K_*(X) \to
K_*(X')$ is an $\ell$-primary torsion group. To see this, let
$F_i'$ be the unique subfield of $F'$ such that $[F_i' : F] = \ell^i$,
and let $X_i'$ be the base-change of $X$ along $\Spec F_i' \to \Spec
F$. The composition of the pull-back and push-forwards maps
$$K_*(X) \to K_*(X_i') \to K_*(X)$$
is equal to multiplication by $\ell^i$, and hence, the kernel of the
pull-back map is annihilated $\ell^i$. Moreover, it follows
from~\cite[Section~IV.8.5]{EGA} that the canonical map
$$\colim_i K_*(X_i') \to K_*(X')$$
is an isomorphism. Since filtered colimits and finite limits of
diagrams of sets commute, we conclude that the kernel of the pull-back
map $K_*(X) \to K_*(X')$ is an $\ell$-primary torsion group as
claimed. It follows that for $q < -d$, $K_q(X)$ is an $\ell$-primary
torsion group. Since this is true for every prime number $\ell$, we
find that for $q < -d$, the group $K_q(X)$ is zero.

We let $X$ be a $d$-dimensional scheme of finite type over an
arbitrary field $k$ of characteristic $p$, and let $k_0 \subset k$ be
a perfect subfield. By~\cite[Theorem~IV.8.8.2]{EGA}, there exists
an intermediate field $k_0 \subset k_1 \subset k$ such that $k_1$ is
finitely generated over $k_0$ together with a scheme $X_1$ of finite
type over $k_1$ such that $X$ is isomorphic to the base-change of
$X_1$ along $\Spec k \to \Spec k_1$. Let $k_1 \subset k_{\alpha}
\subset k$ be a finite generated extension of $k_1$ contained in $k$,
and let $X_{\alpha}$ be the base-change of $X_1$ along $\Spec
k_{\alpha} \to \Spec k_1$. Then $X_{\alpha}$ is of finite type over
$k_{\alpha}$. But then $X_{\alpha}$ is essentially of finite type over
$k_0$. Indeed, the field $k_{\alpha}$ is the quotient field of a ring
$A_{\alpha}$ of finite type over $k_0$, and
by~\cite[Theorem~IV.8.8.2]{EGA}, we can find a scheme
$\mathscr{X}_{\alpha}$ of finite type over $A_{\alpha}$ such that
$X_{\alpha}$ is the generic fiber of $\mathscr{X}_{\alpha}$ over
$A_{\alpha}$. Therefore, the group $K_q(X_{\alpha})$ vanishes for $q <
-d$. Finally, it follows  from~\cite[Proposition~IV.8.5.5]{EGA} that
the canonical map
$$\colim_{\alpha} K_q(X_{\alpha}) \to K_q(X)$$
from the filtered colimit indexed by all intermediate fields $k_1
\subset k_{\alpha} \subset k$ finitely generated over $k_1$ is an
isomorphism. Hence, the group $K_q(X)$ vanishes for $q < -d$ as
stated.
\end{proof}

\begin{acknowledgements}
It is a pleasure to thank Christian Haesemeyer for a number of helpful
conversations. We are particular indebted to Chuck Weibel for pointing
out a mistake in an earlier version to this paper. Finally, we are very
grateful to an anonymous referee for carefully reading the paper and
suggesting several improvements. The work reported in this paper was
done in part while the first author was visiting the University of
Tokyo. He wishes to thank the university and Takeshi Saito in
particular for their hospitality.
\end{acknowledgements}

\providecommand{\bysame}{\leavevmode\hbox to3em{\hrulefill}\thinspace}
\providecommand{\MR}{\relax\ifhmode\unskip\space\fi MR }
\providecommand{\MRhref}[2]{%
  \href{http://www.ams.org/mathscinet-getitem?mr=#1}{#2}
}
\providecommand{\href}[2]{#2}


\begin{thebibliography}{10}

\bibitem{SGA4}
M.~Artin, A.~Grothendieck, and J.~L. Verdier, \emph{Th{\'e}orie des topos et
  cohomologie {\'e}tale des sch{\'e}mas}, S{\'e}minaire de {G\'e}ometrie
  {A}lg{\'e}brique du {B}ois-{Marie} 1963--1964 (SGA 4), Lecture Notes in
  Math., vol. 269, 270, 305, Springer-Verlag, New York, 1972-1973.

\bibitem{blumbergmandell}
A.~J. Blumberg and M.~A. Mandell, \emph{Localization theorems in topological
  {H}ochschild homology and topological cyclic homology}, arXiv:0802.3938.

\bibitem{bokstedthsiangmadsen}
M.~B\"okstedt, W.-C. Hsiang, and I.~Madsen, \emph{The cyclotomic trace and
  algebraic {$K$}-theory of spaces}, Invent. Math. \textbf{111} (1993),
  465--540.

\bibitem{cortinashaesemeyerschlichtingweibel}
G.~Corti\~{n}as, C.~Haesemeyer, M.~Schlichting, and C.~Weibel, \emph{Cyclic
  homology, {$\operatorname{cdh}$}-cohomology, and negative {$K$}-theory}, Ann.
  of Math. \textbf{165} (2007), 1--25.

\bibitem{costeanu}
V.~Costeanu, \emph{On the {$2$}-typical de~{R}ham-{W}itt complex}, Doc. Math.
  \textbf{13} (2008), 413--452.

\bibitem{geisser2}
T.~Geisser, \emph{Arithmetic cohomology over finite fields and special values
  of {$\zeta$}-functions}, Duke Math. J. \textbf{133} (2006), 27--57.

\bibitem{gh5}
T.~Geisser and L.~Hesselholt, \emph{On the relative and bi-relative
  {$K$}-theory of rings of finite characteristic}, arXiv:0810.1337.

\bibitem{gh}
\bysame, \emph{Topological cyclic homology of schemes}, {$K$}-theory (Seattle,
  1997), Proc. Symp. Pure Math., vol.~67, 1999, pp.~41--87.

\bibitem{gros}
M.~Gros, \emph{Classes de {C}hern et classes de cycles en cohomologie de
  {H}odge-{W}itt logarithmique}, M\'{e}moire de la S.~M.~F. \textbf{21} (1985),
  1--87.

\bibitem{EGA}
A.~Grothendieck and J.~Dieudonn{\'{e}}, \emph{{\'E}l{\'e}ments de
  g{\'e}om{\'e}trie alg{\'e}brique}, Inst. Hautes {\'{E}}tudes Sci. Publ. Math.
  \textbf{4, 8, 11, 17, 20, 24, 28, 32} (1960--1967).

\bibitem{haesemeyer}
C.~Haesemeyer, \emph{Descent properties of homotopy {$K$}-theory}, Duke Math.
  J. \textbf{125} (2004), 589--620.

\bibitem{h}
L.~Hesselholt, \emph{On the $p$-typical curves in {Q}uillen's {$K$}-theory},
  Acta Math. \textbf{177} (1997), 1--53.

\bibitem{h8}
\bysame, \emph{Topological {H}ochschild homology and the de~{R}ham-{W}itt
  complex for {$\mathbb{Z}_{(p)}$}-algebras}, Homotopy theory: {R}elations with
  algebraic geometry, group cohomology, and algebraic {$K$}-theory (Evanston,
  IL, 2002), Contemp. Math., vol. 346, Amer. Math. Soc., Providence, RI, 2004,
  pp.~131--145.

\bibitem{hm3}
L.~Hesselholt and I.~Madsen, \emph{On the de~{R}ham-{W}itt complex in mixed
  characteristic}, Ann. Sci. {\'E}cole Norm. Sup. \textbf{37} (2004), 1--43.

\bibitem{hironaka}
H.~Hironaka, \emph{Resolution of singularities of an algebraic variety over a
  field of characteristic zero, {I}, {II}}, Ann. of Math. \textbf{79} (1964),
  109--203; 205--326.

\bibitem{illusie}
L.~Illusie, \emph{Complexe de de {R}ham-{W}itt et cohomologie cristalline},
  Ann. Scient. \'Ec. Norm. Sup. (4) \textbf{12} (1979), 501--661.

\bibitem{katz}
N.~Katz, \emph{Nilpotent connections and the monodromy theorem: applications of
  a result of {T}urrittin}, Publ. Math. I.H.E.S. \textbf{39} (1970), 175--232.

\bibitem{loday}
J.-L. Loday, \emph{Cyclic homology. {S}econd edition}, Grundlehren der
  mathematischen Wissenschaften, vol. 301, Springer-Verlag, New York, 1998.

\bibitem{suslinvoevodsky}
A.~A. Suslin and V.~Voevodsky, \emph{Bloch-{K}ato conjecture and motivic
  cohomology with finite coefficients}, The Arithmetic and Geometry of
  Algebraic Cycles (Banff, Canada, 1998), NATO Sci. Ser. C Math. Phys. Sci.,
  vol. 548, Kluwer, Dordrecht, Germany, 2000, pp.~117--189.

\bibitem{thomason}
R.~W. Thomason, \emph{Algebraic {$K$}-theory and \'etale cohomology}, Ann.
  Scient. \'Ecole Norm. Sup. \textbf{13} (1985), 437--552.

\bibitem{thomasontrobaugh}
R.~W. Thomason and T.~Trobaugh, \emph{Higher algebraic {$K$}-theory of schemes
  and of derived categories}, Grothendieck Festschrift, Volume III, Progress in
  Mathematics, vol.~88, 1990, pp.~247--435.

\bibitem{weibel3}
C.~A. Weibel, \emph{{$K$}-theory and analytic isomorphism}, Invent. Math.
  \textbf{61} (1980), 177--197.

\end{thebibliography}
\end{document}